\documentclass{article}
\usepackage[utf8]{inputenc}

\usepackage{setspace,mathtools,amsmath,amssymb,amsthm,fullpage,outline,centernot,slashed}
\usepackage{bbm}
\usepackage{xcolor}

\usepackage{stmaryrd}
\usepackage{bbm}
\usepackage[french,english]{babel}
\usepackage[T1]{fontenc}
\usepackage{tikz}

\usepackage{setspace,mathtools,amsmath,amssymb,amsthm,fullpage,outline,centernot,slashed} 
\usepackage{mathrsfs} 
\usepackage{enumerate}
\usepackage{accents}
\usepackage[colorlinks=true,citecolor=black,linkcolor=black,urlcolor=blue]{hyperref}
\usepackage{upgreek}

\newcommand{\R}{\mathbb{R}}

\def\bal#1\eal{\begin{align*}#1\end{align*}}
\newtheorem{theorem}{Theorem}

\newtheorem{lemma}[theorem]{Lemma}

\newtheorem{proposition}[theorem]{Proposition}

\newtheorem{definition}{Definition}
\DeclareMathOperator{\divr}{div}

\title{Global stability for nonlinear wave equations satisfying a generalized null condition}
\author{John Anderson\footnote{jrlander@stanford.edu} and Samuel Zbarsky\footnote{zbarskysam@gmail.com}}
\date{\today}

\begin{document}

\maketitle

\begin{abstract}
    We prove global stability for a system of nonlinear wave equations satisfying a generalized null condition. The generalized null condition allows for null forms whose coefficients have bounded $C^k$ norms. We prove both pointwise decay and improved decay of good derivatives using bilinear energy estimates and duality arguments. Combining this strategy with the $r^p$ estimates of Dafermos--Rodnianski then allows us to prove global stability. The proof requires analyzing the geometry of intersecting null hypersurfaces adapted to solutions of wave equations.
\end{abstract}

\tableofcontents

\section{Introduction} \label{sec:intro}

In this paper, we are interested in studying global stability of the trivial solution to nonlinear wave equations satisfying a generalized \emph{null condition}. More precisely, we allow for nonlinear terms which satisfy the classical null condition, but which do not have constant coefficients. A prototypical equation to keep in mind of this form is
\begin{equation} \label{eq:intro}
    \square \phi = -\partial_t^2 \phi + \partial_x^2 \phi + \partial_y^2 \phi + \partial_z^2 \phi = f(t,x) m(d\phi,d\phi)
\end{equation}
where $m$ is the Minkowski metric and $f(t,x)$ is an arbitrary smooth function whose $C^k$ norm is bounded for some large $k$. More generally, we shall allow for linear combinations of classical null forms whose coefficients are allowed to be smooth functions (see Section~\ref{sec:thm}). A rough version of our main Theorem is

\begin{theorem} \label{thm:rough}
The trivial solution to an equation of the form \eqref{eq:intro} in $\R^{3 + 1}$ is globally asymptotically stable relative to sufficiently smooth and small perturbations supported in the unit ball.
\end{theorem}

The precise version of the main Theorem (see Theorem~\ref{thm:main} in Section~\ref{sec:thm}) allows for systems of quasilinear wave equations with nonlinearities coming from our generalized class of null forms.

In the rest of this introduction, we shall begin by describing some background results (Section~\ref{sec:background}) before turning to describe the main difficulties encountered and the strategies used in the proof of the main Theorem (Section~\ref{sec:difficulties} and Section~\ref{sec:proofov}). Section~\ref{sec:notation} describes the notation and various coordinate systems we shall use. Section~\ref{sec:thm} defines the class of null forms we shall consider and provides the precise statement of the main Theorem. The rest of the paper is dedicated to proving the main Theorem. An outline of the rest of the paper is contained in Section~\ref{sec:proofov}.

\subsection{Background and previous results} \label{sec:background}
There is a vast amount of work that has been done on global stability for nonlinear wave equations. We now describe some of these results, focusing on the ones most relevant to the techniques used in this paper, and more generally to the study of wave equations on (potentially nonlinear perturbations of) Minkowski space.

Global stability results for nonlinear wave equations generally use two main ingredients in energy estimates and decay estimates. Energy estimates are well known to control well posedness, and decay estimates can be used to integrate away nonlinear contributions to the energy. One robust way to prove pointwise decay was discovered by Klainerman in the breakthrough work \cite{Kl85}, where he proved a global Sobolev inequality. In $\R^{3 + 1}$, Klainerman's Sobolev inequality reads
\begin{equation} \label{eq:KSineq}
    \begin{aligned}
    |\partial \phi| (t,r,\omega) \le {C \over (1 + t + r) (1 + |t - r|^{{1 \over2}})} \sum_{|\alpha| \le 2} \Vert \partial \Gamma^\alpha \phi \Vert_{L^2 (\Sigma_t)},
    \end{aligned}
\end{equation}
where $\Gamma^\alpha$ is a string of vector fields consisting of translations such as $\partial_t$, rotations such as $x \partial_y - y \partial_x$, boosts such as $t \partial_x + x \partial_t$, and the scaling vector field $t \partial_t + r \partial_r$ (see Section~\ref{sec:notation} for a description of the notation). Because the vector fields in $\Gamma$ satisfy good commutation properties with $\Box$, we expect that estimates satisfied by $\phi$ should roughly be satisfied by $\Gamma \phi$ as well. Thus, the right hand side of \eqref{eq:KSineq} should be bounded by an energy estimate. In \cite{Kl85}, Klainerman was able to use this inequality to prove global stability of wave equations with cubic nonlinearities in $\R^{3 + 1}$. Indeed, by \eqref{eq:KSineq}, appropriate energy boundedness shows decay, and energy boundedness can be recovered using decay to control the error integrals in energy inequalities of the form
\[
\Vert \partial \phi \Vert_{L^2 (\Sigma_s)} \le \Vert \partial \phi \Vert_{L^2 (\Sigma_0)} + \int_0^s \Vert \Box \phi \Vert_{L^2 (\Sigma_t)} d t.
\]

The proof given by Klainerman in \cite{Kl85} is not applicable to quadratic nonlinearities. This restriction is fundamental, as there are examples of nonlinear wave equations in $\R^{3 + 1}$ with quadratic nonlinearities which blow up in finite time (see the work of John \cite{Joh81}). However, in several interesting settings, there is special nonlinear structure that results in global stability. One of the most famous examples is the classical null condition, introduced by Klainerman in \cite{Kla82}. The condition reads as follows.

\begin{definition}
A constant coefficient bilinear form $Q : T \R^{3 + 1} \times T \R^{3 + 1} \rightarrow \R$ is said to satisfy the null condition if $Q(\xi,\xi) = 0$ for any null vector $\xi$, that is, for any vector $\xi$ with $m(\xi,\xi) = 0$ where $m$ is the Minkowski metric.
\end{definition}
While this is stated for semilinear quadratic nonlinearities, we note that there is an analogous condition for quasilinear quadratic nonlinearities (see Section~\ref{sec:thm}).

The null condition leads to improved estimates, resulting in global stability of the trivial solution for quadratic nonlinearities satisfying the null condition, in contrast to the examples studied by John. We can show schematically why nonlinearities satisfying the null condition satisfy improved estimates using the commutation fields from \eqref{eq:KSineq}. If we denote by $\overline{\partial}$ any unit scale derivative tangent to the outgoing null cones $t - r = c$, the null condition guarantees that
\[
|Q(\nabla \phi,\nabla \phi)| \le C |\overline{\partial} \phi| |\partial \phi|,
\]
meaning that at least one of the derivatives must be tangent to the cone. Moreover, we have the algebraic identities
\begin{equation} \label{eq:GoodDerFrameDecomp}
    \begin{aligned}
    |\overline{\partial} \phi| \le {C \over 1 + t + r} \sum_{|\alpha| = 1} |\Gamma^\alpha \phi| \hspace{5 mm} \text{and} \hspace{5 mm} |\partial \phi| \le {C \over 1 + |t - r|} \sum_{|\alpha| = 1} |\Gamma^\alpha \phi|.
    \end{aligned}
\end{equation}
Thus, for solutions to wave equations, we expect that $\overline{\partial}$ derivatives behave better than generic derivatives $\partial$. This means that we expect nonlinearities satisfying the null condition will satisfy improved estimates, as satisfying the null condition guarantees that all ``bad" (generic) derivatives are multiplied by a ``good" (tangential to the light cone) derivative. This has been made precise in several contexts, starting with a proof by Klainerman in \cite{Kla86} of global stability for nonlinear wave equations satisfying the null condition in $\R^{3 + 1}$ (see also the proof by Christodoulou in \cite{Chr86}). Effectively exploiting a kind of null condition present in the Einstein vacuum equations was also a key part in the monumental work \cite{ChrKl93} by Christodoulou--Klainerman showing that Minkowski space is asymptotically stable as a solution to the Einstein vacuum equations. More recently, a generalized null structure was discovered by Lindblad--Rodnianski and used in a series of papers \cite{LinRod03}, \cite{LinRod05}, and \cite{LinRod10}. In these works, Lindblad--Rodnianski provide another proof of the stability of Minkowski space, but in \emph{harmonic coordinates}. This proof involved identifying a weaker version of the null condition, called the \emph{weak null condition}, which is still enough to prove global stability. Indeed, it is only this weaker nonlinear structure which is present in the Einstein vacuum equations when expressed in harmonic gauge. This work was later extended by Lindblad in \cite{Lin08} to other equations satisfying a weak null condition. Finally, in \cite{Kei18}, Keir found a general class of equations satisfying a weak null condition for which both global stability and shock formation ``at infinity'' hold using the $r^p$ estimates of Dafermos--Rodnianski.

In addition to using weighted vector fields as commutators, their use as multipliers has been very important in the study of stability. An example includes the use of weighted energies coming from the conformal Morawetz vector field $K_0 = (t^2 + r^2) \partial_t + 2 t r \partial_r$ (see, for example, \cite{Kla86} and \cite{ChrKl93}). More recently, the celebrated $r^p$ energies of Dafermos--Rodnianski originating in \cite{DafRod10} have been effectively used to study hyperbolic equations on curved backgrounds.

Other approaches to global stability have been successfully used in several interesting contexts recently as well. One of these approaches, the method of \emph{spacetime resonances} (see, for example, \cite{GerMasSha09}), involves writing the solution to the nonlinear equation by applying the Duhamel formula corresponding to the linear propagator to the nonlinearity. The nonlinearity can then be controlled using the propagator to show that certain frequencies interact less because they either oscillate at different rates in time (the transformation that exploits this is called a \emph{normal form}, see \cite{Sim83} and \cite{Sha85}), or because they have different group velocities, meaning that they separate in physical space. See \cite{Ger11} for an introduction to this method. For hyperbolic equations, we mention in particular the work of Pusateri--Shatah \cite{PusSha13} and the work of Deng--Pusateri \cite{DenPus20}. In \cite{PusSha13}, Pusateri--Shatah understand the classical null condition in frequency space and provide another proof of global stability of the trivial solution. They in fact allow for a generalized class of null forms which includes zeroth order operators that still have null structure. Meanwhile, the work \cite{DenPus20} provides a proof using spacetime resonances of global stability for a class of wave equations satisfying the weak null condition of Lindblad--Rodnianski.

Another approach has involved using wave packets to prove decay, used by Ifrim--Tataru in a nonlinear setting in \cite{IfrTat15} (see also the references therein). The solution can be controlled by taking inner products with wave packet approximate solutions, and controlling how these inner products change in time. Controlling the inner products then gives control over the solution, as it effectively writes the solution as a sum of $L^2$ orthogonal wave packets.

These approaches have similarities in spirit to the strategy we follow in this paper (see also \cite{JohKla84} and \cite{MetTatToh12}, which encounter similar considerations through the use of the fundamental solution of the wave equation). We shall extend the strategy used in \cite{And21}, which proves decay using bilinear energy estimates and testing against auxiliary solutions to the homogeneous wave equation. Decay then follows from a duality argument which uses pointwise estimates for solutions to the homogeneous wave equation as a black box. This strategy is modified in this paper to additional show improved decay of good derivatives. Combining these pointwise estimates with weighted energies coming from the $r^p$ estimates of Dafermos--Rodnianski will then allow us to close a bootstrap argument and prove global stability. This shall be described in more detail in Section~\ref{sec:proofov}.



There are several situations in which the problem limits our access to weighted commutators. The particular problem we are studying has this feature, and we now turn to describing the main difficulties and how we will overcome them in Section~\ref{sec:difficulties} and Section~\ref{sec:proofov}.

\subsection{Main difficulties} \label{sec:difficulties}
We recall the equation from \eqref{eq:intro} given by
\[
\Box \phi = f(t,x) m(d \phi,d \phi),
\]
where $f$ is an arbitrary smooth function with bounded $C^k$ norm. Examining Klainerman's Sobolev inequality \eqref{eq:KSineq}, we see that we no longer have easy access to any weighted vector fields as commutators because they introduce bad weights as soon as they fall on the function $f$. Thus, the first difficulty we must overcome is proving pointwise decay without access to any weighted commutators. We must also find a way of taking advantage of the null condition without access to weighted commutators.

Another difficulty is that it seems as though the introduction of the nontrivial function $f$ reduces the pointwise decay in $u$. Indeed, one can show the following statement: Given any point $(t_0,x_0)$ in $\R^{3 + 1}$ with $t_0 \ge |x_0|+10$, there exists a smooth function $F$ supported in the set of points where $|t - r| \le 2$ such that $|F| \le {C \over (1 + t)^3}$ and the solution $\phi$ to the equation $\Box \phi = F$ with vanishing initial data has that $|\partial \phi(t_0,x_0)| \gtrsim {1 \over (1 + t) (1 + |t - |x_0||)}$. Because of this, we expect that we cannot show integrable in $u$ decay.

\subsection{Proof overview} \label{sec:proofov}
The reader may wish to consult Section~\ref{sec:notation} for a description of the notation we shall use while reading this overview.

In order to overcome the difficulties described above, the two main things we must understand is how to take advantage of the null condition and how to prove pointwise decay estimates. Recalling the schematic equation \eqref{eq:intro} given by
\[
\Box \phi = f(t,x) m(d \phi,d \phi),
\]
we first note that, because the term on the right satisfies the null condition, we have that
\[
|\Box \phi| \le C |\partial \phi| |\overline{\partial} \phi|.
\]
We are afraid of commuting with weighted vector fields because of the possibility of them falling on $f$, but we may still commute with translation vector fields. Doing so, we see that
\[
|\partial^\alpha \Box \phi| = \left |\sum_{\mu+\beta + \gamma = \alpha} (\partial^\mu f) m(d \partial^\beta \phi,d \partial^\gamma \phi) \right | \le C \Vert f \Vert_{C^k} \sum_{\mu+\beta + \gamma = \alpha} (|\partial \partial^\beta \phi| |\overline{\partial} \partial^\gamma \phi|)
\]
for $|\alpha| \le k$. This means that the null structure guaranteeing that we can write things as a product of a good and bad derivative is stable under commuting with translation vector fields. We shall then follow the usual strategy of proving stability by proving pointwise decay of general derivatives and improved decay of good derivatives. However, as was described in Section~\ref{sec:difficulties}, we must do so using only translation vector fields. The fact that the null forms we consider are well behaved under commutation with translation vector fields follows almost immediately as in the above and from the definition of our class of null forms (see Section~\ref{sec:thm} for the definition of the class of null forms, and see Section~\ref{sec:nullforms} for a proof of this property for this class of null forms).

We now describe how we show pointwise decay of general derivatives and improved decay of good derivatives. In order to show pointwise decay of general derivatives, we shall essentially use the method from \cite{And21} using bilinear energy estimates. In order to show improved decay of good derivatives, we shall introduce a way to do so using bilinear energy estimates.

We first describe the strategy for showing pointwise decay of general derivatives. We shall use bilinear energy estimates in order to test the solution $\phi$ of the nonlinear equation against solutions $\psi$ of the homogeneous wave equation $\Box \psi = 0$. Choosing the data for $\psi$ well will then give us estimates on $\phi$. To see how this works, we can multiply the equation for $\phi$ by $\partial_t \psi$ and integrate by parts between two constant time hypersurfaces, $\Sigma_s$ and $\Sigma_0$, one of which contains the data for $\phi$. This is analogous to the usual energy estimate. This results in the identity
\begin{equation}
    \begin{aligned}
    \int_{\Sigma_s} (\partial_t \phi) (\partial_t \psi) + (\partial^i \phi) (\partial_i \psi) d x = \int_{\Sigma_0} (\partial_t \phi) (\partial_t \psi) + (\partial^i \phi) (\partial_i \psi) d x - \int_0^s \int_{\Sigma_t} (\Box \phi) (\partial_t \psi) d x d t,
    \end{aligned}
\end{equation}
where we are using the fact that $\Box \psi = 0$. Now, if we are interested in studying $\phi$ near some point with coordinates $(s,x_0)$, it is natural to take data for $\psi$ supported in the unit ball of radius, say, $1$ centered at $x_0$ within $\Sigma_s$. In this case, the identity reads
\begin{equation}
    \begin{aligned}
    \int_{B_1 (s,x_0)} (\partial_t \phi) (\partial_t \psi) + (\partial^i \phi) (\partial_i \psi) d x = \int_{\Sigma_0} (\partial_t \phi) (\partial_t \psi) + (\partial^i \phi) (\partial_i \psi) d x - \int_0^s \int_{\Sigma_t} (\Box \phi) (\partial_t \psi) d x d t,
    \end{aligned}
\end{equation}
where we are using $B_1 (s,x_0)$ to denote the ball of radius $1$ centered at $x_0$ in $\Sigma_s$. The term on the left hand side involves the data for $\psi$ and the solution we wish to study $\phi$. Meanwhile, the first term on the right hand involves the data for $\phi$ and the solution $\psi$ to the homogeneous wave equation. Thus, if we allow the data for $\psi$ to vary in a sufficiently large class, and if we are able to control solutions to the homogeneous wave equation arising from these data, we can control the solution at $(s,x_0)$ by a duality argument, and we can control the integral
\[
\int_{\Sigma_0} (\partial_t \phi) (\partial_t \psi) + (\partial^i \phi) (\partial_i \psi) d x
\]
in terms of data for $\phi$ and the estimates we have on $\psi$. In practice, the error integral involving $\Box \phi$ is nonlinear, but we can hope to control this in the context of a bootstrap argument using the estimates we have on $\psi$.

In order to show improved decay of good derivatives, we shall use different multipliers in the bilinear energy estimates. The choice of multipliers is geometrically motivated (the reader may wish to consult Figure~\ref{fig:lightconesaux} and Figure~\ref{fig:lightconesplane} throughout this discussion). Just as before, if we are interested in studying the behavior of $\phi$ at some point $(s,x_0)$, it is natural to consider data for $\psi$ supported near this point. To have a relatively clear geometric picture, it is helpful to think that $s - |x_0|$ is not too large. Moreover, we shall assume that the point has $y$ and $z$ coordinate equal to $0$.

The good derivatives for $\phi$ at this point are given by the rotation vector fields scaled to have length $1$ and by the vector field $\partial_t + \partial_r$. Around the point in question, the rescaled rotation vector field ${1 \over r} (y \partial_z - z \partial_y)$ approximately vanishes, so we shall only consider the other ones. This leaves us with the vector fields $\partial_t + \partial_r$, ${1 \over r} (x \partial_y - y \partial_x)$, and ${1 \over r} (x \partial_z - z \partial_y)$. However, because $y$ and $z$ are approximately $0$ near this point, we can approximate these two vector fields by ${x \over r} \partial_y \approx \partial_y$ and ${x \over r} \partial_z \approx \partial_z$. Similarly, we can see that $\partial_t + \partial_r \approx \partial_t + \partial_x$. This all amounts to approximating the null cone in question with the null hyperplane tangent to it along the null ray in the $t$, $x$ plane (see Figure~\ref{fig:lightconesplane}).

The reason we have made these reductions is that $\partial_y$, $\partial_z$, and $\partial_t + \partial_x$ are all Killing vector fields of Minkowski space. Thus, they are well behaved as multipliers. Working, for example, with $\partial_t + \partial_x$ and using $(\partial_t + \partial_x) \psi$ as a multiplier in the equation for $\phi$ gives us that
\begin{equation}
    \begin{aligned}
    \int_{B_1 (s,x_0)} (\partial_t \phi + \partial_x \phi) (\partial_t \psi + \partial_x \psi) + \sum_{j \ne 1} (\partial_j \phi) (\partial_j \psi) d x = \int_{\Sigma_0} (\partial_t \phi + \partial_x \phi) (\partial_t \psi + \partial_x \psi) + \sum_{j \ne 1} (\partial_j \phi) (\partial_j \psi) d x
    \\ - \int_0^s \int_{\Sigma_t} (\Box \phi) (\partial_t \psi + \partial_x \psi) d x d t.
    \end{aligned}
\end{equation}
Once again, allowing the data for $\psi$ to vary in a certain class will allow us to prove estimates on $\phi$. Moreover, on the left hand side, the expression for $(\partial_t + \partial_x) \phi$ will allow us to appropriately pick data for $\psi$ to control only this derivative of $\phi$.

The estimates this will give on $(\partial_t + \partial_x) \phi$ are better than the ones for a generic derivative. In order to see that this is the case, we first look at the error integral involving $\Box \phi$. This integral contains a factor of $(\partial_t + \partial_x) \psi$. Looking geometrically, we can see that $(\partial_t + \partial_x) \psi$ is approximately tangent to a light cone adapted for $\psi$, meaning that it is approximately a good derivative of $\psi$. Thus, it should decay more quickly, meaning that the integral should be better behaved. Similarly, in the integral involving data, the $(\partial_t + \partial_x) \psi$ term is once again approximately a good derivative for $\psi$. The other terms include $\partial_y \psi$ and $\partial_z \psi$, but we can see that these are approximately tangent to the light cone for $\psi$ as well (they are almost equal to rescaled rotation fields for $\psi$). Thus, this term should be better as well. Altogether, these facts will allow us to show improved decay for $(\partial_t + \partial_x) \phi$ at this particular point. Because $\partial_t + \partial_x \approx \partial_t + \partial_r$, this is showing improved decay for $\partial_t + \partial_r$, as desired. A similar discussion holds for the rescaled rotation vector fields and the translation vector fields in $y$ and $z$. This strategy will allow us to show improved pointwise decay for the good derivatives of $\phi$. The results needed to prove pointwise decay statements (both for generic derivatives and improved decay for good derivatives) using bilinear energy estimates will be established in Section~\ref{sec:auxmultiplier}. Meanwhile, the geometry will be studied in Section~\ref{sec:geometry}.

We now summarize the strategy just described for proving decay. The first tool is a bilinear energy estimate giving integration by parts identities involving two solutions of wave equations. One of the two solutions is in practice the solution we want to control. The other is allowed to vary. More precisely, we allow it to vary in such a way that we get estimates at a particular place in physical space. Showing improved decay gives a finer understanding of phase space, as it involves only certain derivatives in certain regions. In order to show this, we pick the multipliers based on the geometry a bit more carefully. This allows us to get improved control over specific derivatives, as expected. Finally, we must control the nonlinear error integral, which is the hardest step. Implementing this strategy requires good multipliers for (both regular and bilinear) energy estimates, pointwise estimates on solutions of the homogeneous wave equation, and estimates on the geometry of null cones. For simplicity, we shall use pointwise estimates coming from the fundamental solution, but this is not necessary, and any preferred method of proving linear estimates can be used, such as stationary phase or vector fields.

There is one further difficulty while following this strategy which we must mention. We are only able to propagate a decay rate of ${1 \over t u}$ (up to a small loss). It is, in particular, not integrable in $u$, and this seems to be sharp for the worst case pointwise behavior because of the function $f$ in \eqref{eq:intro}. This introduces technical difficulties. However, the solution ought to be better behaved than this on average. In order to take advantage of this fact, we shall use the weighted $r^p$ energies of Dafermos--Rodnianski from \cite{DafRod10}. This will allow us to prove that the energy decays in $u$ in a way that shows that the solution is better behaved than the above pointwise decay rate on average. We shall, in particular, use these arguments to prove weighted energy estimates for $\phi$ on null cones adapted to the auxiliary solution to the homogeneous wave equation $\psi$. These considerations are developed in Section~\ref{sec:linests}.

With all of these things in hand, the proof of the main Theorem, Theorem~\ref{thm:main}, will be completed in Section~\ref{sec:bootstrap} using a bootstrap argument.

\subsection{Acknowledgements}
We are extremely grateful to our advisors, Sergiu Klainerman and Igor Rodnianski, for several helpful discussions. We are also very grateful to Yakov Shlapentokh-Rothman for making us aware of this problem. JA gratefully acknowledges that this material is based upon work partly supported by the National Science Foundation under Grant No. 2103266. SZ's work was partially supported by the National Science Foundation Graduate Research Fellowship Program under Grant No. DGE-1656466. Any opinions, findings, and conclusions or recommendations expressed in this material are those of the authors and do not necessarily reflect the views of the National Science Foundation.

\section{Notation and coordinates} \label{sec:notation}
We are working in $\R^{3 + 1}$ with rectangular coordinates $(t,x,y,z)$, and equipped with the Minkowski metric $m = -d t^2 + d x^2 + d y^2 + d z^2$. We denote by $(t,r,\omega)$ with $\omega \in S^2$ polar coordinates adapted to these rectangular coordinates, where $r^2 = x^2 + y^2 + z^2$. We shall also denote by $v$ and $u$ the usual families of null cones adapted to this setting, that is, $v = t + r$ and $u = t - r$.

As was described in Section~\ref{sec:proofov}, we shall prove appropriate decay estimates (found in Section~\ref{sec:auxmultiplier}) by using auxiliary solutions $\psi$ of the homogeneous wave equation as multipliers. The data for $\psi$ are specified in some ball contained within some constant $t$ hypersurface. We shall denote by $s$ the time coordinate corresponding to this time slice. Thus, we can denote the spacetime coordinates of the center of the ball where the data for $\psi$ are specified as being $(s,x_0,y_0,z_0)$. The $u$ coordinate of this point will be very important in the following, and we shall denote it by $\tau$, that is,
\[
\tau := s - \sqrt{x_0^2 + y_0^2 + z_0^2}.
\]

We shall also need coordinates adapted to the auxiliary multiplier. We thus take rectangular coordinates adapted to $\psi$ given by $t' = s - t$, $x' = x - x_0$, $y' = y - y_0$, and $z' = z - z_0$. We then denote by $r'$ the radial coordinate adapted to $\psi$, that is,
\[
(r')^2 = (x')^2 + (y')^2 + (z')^2 = (x - x_0)^2 + (y - y_0)^2 + (z - z_0)^2.
\]
We similarly have the null cones parameterized by $v' = s - t + r'$ and $u' = s - t - r'$. We note that we have $s - t$ in place of $t$ because we should think of $\Sigma_s$ as being the initial time slice for $\psi$, and we think of it as evolving back in time.

We shall also need to track the behavior of good derivatives versus generic derivatives for both the solution $\phi$ and the auxiliary multiplier $\phi$. We shall do so using a kind of \emph{null frame}, which effectively encodes this information (see, for example, \cite{Ali10}).

To this end, we define the vector fields
\[
L = \partial_t + \partial_r, \hspace{5 mm} \underline{L} = \partial_t - \partial_r,
\]
and, analogously for $\psi$,
\[
L' = -\partial_t + \partial_{r'}, \hspace{5 mm} \underline{L}' = -\partial_t - \partial_{r'}.
\]
We then complete this with unit scale angular derivatives, $e_A$ for $\phi$ and $e_A'$ for $\psi$. The $e_A$ are tangent to the level sets of $v$ and $u$ while the $e_A'$ are tangent to the level sets of $v'$ and $u'$. Moreover, the $e_A$ are orthogonal to $L$ and $\underline{L}$ while the $e_A'$ are orthogonal to $L'$ and $\underline{L}'$. We take these angular derivatives to be angular vector fields that are appropriately rescaled by $r$. Thus, we denote by $e_A$ one of the vector fields
\[
e_{x y} = {1 \over r} \Omega_{x y} = {1 \over r} (x \partial_y - y \partial_x), \hspace{2 mm} e_{x z} = {1 \over r} \Omega_{x z} = {1 \over r} (x \partial_z - z \partial_x), \hspace{2 mm} e_{y z} = {1 \over r} \Omega_{y z} = {1 \over r} (y \partial_z - z \partial_y),
\]
and similarly, we denote by $e_A'$ one of the vector fields
\[
e_{x' y'} = {1 \over r'} \Omega_{x' y'} = {1 \over r'} (x' \partial_{y'} - y' \partial_{x'}), \hspace{2 mm} e_{x' z'} = {1 \over r'} \Omega_{x' z'} = {1 \over r'} (x' \partial_{z'} - z' \partial_{x'}), \hspace{2 mm} e_{y' z'} = {1 \over r'} \Omega_{y' z'} = {1 \over r'} (y' \partial_{z'} - z' \partial_{y'}).
\]
We then define the \emph{good derivatives} of $\phi$ to be those consisting of $L$ and $e_A$, and similarly, we define the good derivatives of $\psi$ to be those consisting of $L'$ and $e_A'$. We shall denote by $\overline{\partial}$ the good derivatives of $\phi$ and by $\overline{\partial}'$ the good derivatives of $\psi$.

We shall need to compute various changes of frame for the above families. In order to do so, it will sometimes be convenient to use the Euclidean inner product. Given two vectors $v$ and $w$ in $T_p \R^{2 + 1}$, we shall denote their Euclidean inner product by $\langle v,w \rangle_e$.

We shall need a few more geometric quantities adapted to both the solution $\phi$ and the auxiliary multiplier $\psi$. Without loss of generality (by rotating the coordinates), we note that we can assume that $y_0 = 0$, $z_0 = 0$, and $x_0 > 0$. Thus, we shall often assume that the spacetime coordinates of this point are given by $(s,a,0,0)$ with $a > 0$. In this formulation, we have that $\tau = s - a$, and that $r' = (x - a)^2 + y^2 + z^2$. With this simplification, we define $\rho^2 = y^2 + z^2$. Within each $\Sigma_t$, the level sets of $\rho$ are cylinders with axis given by the $x$ axis. We complete this into a system of cylindrical coordinates in $(t,x,y,z)$ space, given by $(t,x,\rho,\theta)$ where $\rho$ is as in the above and
\[
\theta = \arccos \left ({y \over \rho} \right ).
\]
We also set
\[
\vartheta = \arccos \left ({x \over r} \right ) \hspace{5 mm} \text{and} \hspace{5 mm} \vartheta' = \arccos \left ({a - x \over r'} \right ).
\]
We then note that $\rho = r \sin(\vartheta) = r' \sin(\vartheta')$. See Figure~\ref{fig:rholowerbound} in Section~\ref{sec:geometry} for a picture describing these quantities.

We shall also use vector fields adapted to null planes which approximate the null cones of the solution and the auxiliary multiplier. With the above simplifications, this corresponds to considering the null planes which are level sets of $\hat{u} = t - x$. The vector fields $\partial_t + \partial_x$, $\partial_y$, and $\partial_z$ tangent to the level sets of $\hat{u}$ will be denoted by $\hat{\partial}$. These approximate the good derivatives of the solution and auxiliary multiplier in appropriate regimes (i.e., when the null plane is a good approximation of the null cone).

We shall also need to introduce other coordinate systems more carefully adapted to the light cones for $\phi$ and $\psi$. These coordinate systems are given by $(t,r,r',\theta)$ and $(t,u,u',\theta)$ where $\theta$ denotes an angular coordinate adapted to the circle at the intersection of fixed $r$, $r'$, and $t$. More precisely, the first coordinate system is defined when the intersection of level sets of $t$, $r$, and $r'$ is nonempty. When this is the case, the intersection is either a circle or a single point. Intersections consisting of a single point occur on a set of measure $0$ (such intersection points in fact consist of only the $x$ axis). When the intersection is instead a circle (which we note has its center on the $x$ axis), we are denoting by $\theta$ an angular coordinate for this circle. Analogous properties hold for the second coordinate system. The volume form in in these coordinates can be found in Section~\ref{sec:geometry}.

We finally discuss the parameters $\epsilon$, $\delta$, and $N_0$, as well as the dependence of the constants $C$ which appear throughout the paper. The parameter $\delta<1/100$ comes from the bootstrap assumptions, and it also controls losses that come from interpolating between energy bounds and pointwise decay estimates (see Section~\ref{sec:BootstrapAssumptions}). The parameter $N_0$ determines the Sobolev regularity on the initial data and the solution, and also on the coefficients we allow on the null forms (see Section~\ref{sec:thm}). Because $N_0$ determines interpolation losses and because we assume that the interpolation losses are small in terms of $\delta$ in Section~\ref{sec:BootstrapAssumptions}, we note that $N_0$ is allowed to depend on $\delta$. The constant $C$ is allowed to depend on $\delta$ and $N_0$, but not on $\epsilon$. The parameter $\epsilon$ measures the size of the initial data, and we are allowed to choose it sufficiently small in terms of all of the other parameters in order to make the proof work.

\section{Main Theorem} \label{sec:thm}
We begin by describing the admissible nonlinearities which satisfy our generalized null condition. The condition is very simple, allowing for arbitrary linear combinations of classical null forms, where the coefficients in the linear combination are allowed to be smooth functions whose $C^k$ norms are bounded for $k$ sufficiently large. A simple example of such a null form is the one found in \eqref{eq:intro}, where the function $f$ is the coefficient.. More precisely, we impose the following condition on quadratic nonlinearities.

\begin{definition} \label{def:quadnull}
We say that a bilinear form $Q_2 : T \R^{3 + 1} \times T \R^{3 + 1} \rightarrow \R$ is admissible with regularity $N_0$ if $Q_2 (\xi,\xi) = 0$ for any null vector $\xi$, and if the Cartesian components $Q_2 (\partial^\alpha,\partial^\beta)$ have uniformly bounded $C^{N_0}$ norm. Similarly, we say that a trilinear form $Q_3 : T \R^{3 + 1} \times T \R^{3 + 1} \times T \R^{3 + 1} \rightarrow \R$ is admissible with regularity $N_0$ if $Q_3 (\xi,\xi,\xi) = 0$ for any null vector $\xi$, and if the Cartesian components $Q_3 (\partial^\alpha,\partial^\beta,\partial^\gamma)$ have uniformly bounded $C^{N_0}$ norm.
\end{definition}

This generalizes the null condition because it allows for the coefficients to not be constant as long as they are bounded in $C^{N_0}$. Because these are simply tensors on Minkowski space, we note that this definition is equivalent to being a linear combination of classical null forms where the coefficients are allowed to be smooth functions whose $C^{N_0}$ norm are bounded.

We also wish to allow for higher order nonlinearities with varying coefficients. For these expressions, we have

\begin{definition} \label{def:higherordernon}
Let $R$ be a multilinear form on $\R^{3 + 1}$. We say that $R$ is an admissible with regularity $N_0$ higher order nonlinearity if its Cartesian components have uniformly bounded $C^{N_0}$ norm.
\end{definition}

We can now give a precise statement of the main Theorem. In general, we consider quasilinear systems of $N$ equations and $N$ unknowns in $\R^{3 + 1}$. Allowing capital Latin indices to run from $1$ to $N$ and following the Einstein summation convention in these indices, we can write the equations in the form
\begin{equation} \label{eq:main}
    \begin{aligned}
        \Box \phi^I = (Q_2)^I_{J K} (d \phi^J,d \phi^K) + (Q_3)^I_{J K} (d^2 \phi^J,d \phi^K) + R^I_{J K} (d^2 \phi^J,d \phi^K),
    \\ \phi^I (0,x) = \phi_0^I (x),
    \\ \partial_t \phi^I (0,x) = \phi_1^I (x).
    \end{aligned}
\end{equation}
Here, each $(Q_2)^I_{J K}$ and $(Q_3)^I_{J K}$ denotes a bilinear and trilinear forms, respectively, which encode the quadratic nonlinearities. Similarly, each $R^I_{J K}$ denotes a higher order nonlinearity, meaning that it vanishes to at least third order when $d^2 \phi$ and $d \phi$ vanish. With this at hand, we can now state the main Theorem.

\begin{theorem} \label{thm:main}
Let $N_0$ be sufficiently large, and let the $(Q_2)^I_{J K}$, $(Q_3)^I_{J K}$, and $R^I_{J K}$ satisfy the admissibility conditions in Definitions~\ref{def:quadnull} and \ref{def:higherordernon} with regularity $N_0$. Then, there exists an $\epsilon_0 > 0$ such that the trivial solution $\phi^I = 0$ of \eqref{eq:main} is globally asymptotically stable relative to perturbations $(\phi_0^I,\phi_1^I)$ which are compactly supported in the unit ball and which have that $\Vert \phi_0^I \Vert_{H^{N_0 + 1}} + \Vert \phi_1^I \Vert_{H^{N_0}} \le \epsilon < \epsilon_0$.
\end{theorem}
As was described in Section~\ref{sec:proofov}, we shall prove the pointwise decay estimates required for the global stability result using bilinear energy estimates and duality. Improved decay for good derivatives will require us to use the phase space behavior of solutions to wave equations in a more refined way than proving pointwise decay for general derivatives. These considerations may be of interest in other situations where, for example, we do not have access to many weighted commutators.

Before proceeding, we shall immediately make the technical simplification of assuming that our system is instead a scalar equation, and we shall ignore both $Q_3$ and $R$. Thus, we shall only consider the equation
\begin{equation} \label{eq:mainsimple}
    \begin{aligned}
        \Box \phi = Q_2 (d \phi,d \phi) = f(t,x) m(d \phi,d \phi),
    \\ \phi (0,x) = \phi_0 (x),
    \\ \partial_t \phi (0,x) = \phi_1 (x).
    \end{aligned}
\end{equation}
This greatly simplifies the notation. The proof will immediately generalize to systems and more general null forms because the estimates we shall use are satisfied by all null forms (in particular, we shall not use the fact that the metric null form satisfies improved estimates). Moreover, terms coming from $R$ are easier to handle. In fact, they were already treated in a weaker fashion in \cite{And21}. Combining the arguments there and the ones here can control these terms.

\section{Properties of the generalized null forms} \label{sec:nullforms}
We now list the main property of the null forms we consider which allows us to prove global stability.
\begin{lemma} \label{lem:nullforms}
Let $Q_2$, $Q_3$, and $R$ denote an admissible with regularity $N_0$ bilinear null form, trilinear null form, and multilinear form, respectively, from Section~\ref{sec:thm}. Moreover, let $h_1$ and $h_2$ denote arbitrary smooth functions. Then, we have that
\begin{enumerate}
    \item For all $|\alpha| \le N_0$, we have that
    \[
    |\partial^\alpha Q_2 (d h_1,d h_2)| \le C\sum_{|\beta|+|\gamma|\le N_0} (|\overline{\partial}\partial^\beta h_1| |\partial\partial^\gamma h_2| + |\overline{\partial}\partial^\gamma h_2| |\partial\partial^\beta h_1|).
    \]
    \item For all $|\alpha| \le N_0$, we have that
    \[
    |\partial^\alpha Q_3 (d^2 h_1,d h_2)| \le C\sum_{|\beta|+|\gamma|\le N_0} (|\overline{\partial} \partial \partial^\beta h_1| |\partial\partial^\gamma h_2| + |\partial^2 \partial^\beta h_1| |\overline{\partial} \partial^\gamma h_2|).
    \]
    \item If $R$ additionally vanishes to at least third order when its entries vanish, then, for all $|\alpha| \le N_0$, we have that
    \[
    |\partial^\alpha R(d^2 h_1,d h_2)| \le C\sum_{|\beta|+|\gamma|\le N_0} (|\partial^2\partial^\beta h_1| + |\partial\partial^\gamma h_2|)^3.
    \]
\end{enumerate}
\end{lemma}
\begin{proof}
This follows from the analogous result for constant coefficient forms, the definition of the admissible forms, and the observation that the admissible null forms can be written as linear combinations of classical null forms whose coefficients are sufficiently smooth functions.
\end{proof}
Thus, while classical null forms are stable after commuting with translations and the usual weighted commutators, the generalized null forms are stable when commuting with translation vector fields. This is the main difference between the classical null forms and the ones considered here. For the null forms we consider, we can only guarantee the product structure in terms of good and bad derivatives after commuting with translation vector fields. This fact will still be very important in the proof.

Before proceeding, we note that what is truly necessary for the proof of Theorem~\ref{thm:main} is weaker than what is provided by Lemma~\ref{lem:nullforms}. Indeed, the scheme has space in controlling the nonlinear terms, so we can allow for the coefficients of the null form to grow at some small polynomial rate in $t$ or $r$ (or both). In fact, we can allow the growth to become slightly worse with every successive differentiation, so behavior like $\sin(\langle t \rangle^\alpha x)$ for some $\alpha > 0$ is also admissible for the coefficients of the null forms. For sufficiently small such growths, the proof carries through as is.

\section{Proving decay using auxiliary multipliers} \label{sec:auxmultiplier}
In this Section, we shall use bilinear energy estimates to derive inequalities which will give us pointwise decay in the proof of Theorem~\ref{thm:main}. The bilinear energy estimates will lead to pointwise decay estimates by a duality argument. In Section~\ref{sec:bilinearen}, we establish the particular bilinear energy estimates that we shall use. Then, we show how these bilinear energy estimates give us an inequality that implies pointwise decay for general derivatives (Section~\ref{sec:badder}) and improved decay for good derivatives (Section~\ref{sec:goodder}).

\subsection{Bilinear energy estimates} \label{sec:bilinearen}
We consider the bilinear energy momentum tensor
\[
T [\psi,\phi] = d \psi \otimes d \phi + d \phi \otimes d \psi - g(d \psi,d \phi) g.
\]
We recall that
\[
\divr(T) = (\Box \psi) d \phi + (\Box \phi) d \psi.
\]
Contracting with appropriate vector fields and applying the divergence theorem will thus give us conserved quantities, with $\Box \psi$ and $\Box \phi$ giving us error terms. We list various integration by parts identities we will use in the following proposition.

\begin{proposition} \label{prop:bilinen}
Let $\psi$ and $\phi$ be smooth functions whose traces at $t = 0$ are compactly supported.
\begin{enumerate}
    \item We have that
    \begin{equation}
        \begin{aligned}
            \int_{\Sigma_s} (\partial_t \psi) (\partial_t \phi) + (\partial^i \psi) (\partial_i \phi) d x + \int_0^s \int_{\Sigma_t} (\Box \psi) (\partial_t \phi) + (\Box \phi) (\partial_t \psi) d x d t = \int_{\Sigma_0} (\partial_t \psi) (\partial_t \phi) + (\partial^i \psi) (\partial_i \phi) d x.
        \end{aligned}
    \end{equation}
    \item We have that
    \begin{equation}
        \begin{aligned}
            \int_{\Sigma_s} (\partial_t \psi) (\partial_i \phi) + (\partial_i \psi) (\partial_t \phi) d x + \int_0^s \int_{\Sigma_t} (\Box \psi) (\partial_i \phi) + (\Box \phi) (\partial_i \psi) d x d t = \int_{\Sigma_0} (\partial_t \psi) (\partial_i \phi) + (\partial_i \psi) (\partial_t \phi) d x.
        \end{aligned}
    \end{equation}
    \item We have that
    \begin{equation}
        \begin{aligned}
            \int_{\Sigma_s} (\partial_t \psi + \partial_i \psi) (\partial_t \phi + \partial_i \phi) + \sum_{j \ne i} (\partial_j \psi) (\partial_j \phi) d x + \int_0^s \int_{\Sigma_t} (\Box \psi) (\partial_t \phi + \partial_i \phi) + (\Box \phi) (\partial_t \psi + \partial_i \psi) d x d t
            \\ = \int_{\Sigma_0} (\partial_t \psi + \partial_i \psi) (\partial_t \phi + \partial_i \phi) + \sum_{j \ne i} (\partial_j \psi) (\partial_j \phi) d x.
        \end{aligned}
    \end{equation}
\end{enumerate}
\end{proposition}
\begin{proof}
These formulas follow from contracting $T$ with an appropriate translation vector field $V$ ($V = \partial_t$ for the first, $V = \partial_i$ for the second, and $V = \partial_t + \partial_i$ for the third), taking the divergence, and using the divergence theorem on the resulting expression in the region between $\Sigma_0$ and $\Sigma_s$. We note that this uses the fact that translation vector fields are Killing vector fields of Minkowski space. Alternatively, this can also be shown by multiplying $\Box \psi$ by $V \phi$ (where $V$ is as above for each of the three cases) and integrating by parts.
\end{proof}

\subsection{Pointwise decay of a generic derivative} \label{sec:badder}
We now use Proposition~\ref{prop:bilinen} to prove decay estimates on $\phi$ and $\partial \phi$. We take $\psi$ to solve the homogeneous wave equation and choose its data in terms of where we want to get estimates for $\phi$, as is described in Section~\ref{sec:proofov}.

Suppose we want to show decay for $\phi$ at a point $p$ whose coordinates are $(s,x_0,y_0,z_0)$. Recalling the discussion in Section~\ref{sec:notation}, we denote by $\tau$ the $u$ coordinate of this point, that is,
\[
\tau = s - \sqrt{x_0^2 + y_0^2 + z_0^2}.
\]
We then consider the ball $B_1 (p)$ of radius $1$ centered at $p$ in $\Sigma_s$, and we only consider data $(\psi_0,\psi_1)$ for $\psi$ whose $C^5$ norm are bounded by $10$ and which are smooth and compactly supported in this ball. By appropriately choosing the data for $\psi$, commuting with translation vector fields, and using a duality argument (see \cite{And21} where this is worked out in detail), we can get control over the first derivatives of $\phi$.

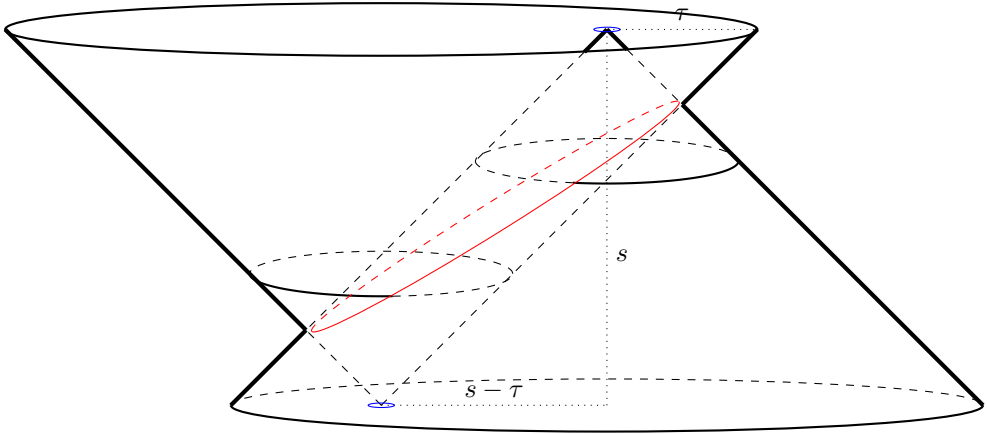
\begin{figure}
    \centering
    \begin{tikzpicture}
    \draw[dashed] (0,0) -- (4,4);
    \draw[ultra thick] (4,4) -- (5,5);
    \draw[dashed] (0,0) -- (-1,1);
    \draw[ultra thick] (-1,1) -- (-5,5);
    \draw[ultra thick] (3,5) -- (2.7,4.7);
    \draw[dashed] (2.7,4.7) -- (-1,1);
    \draw[ultra thick] (-1,1) -- (-2,0);
    \draw[ultra thick] (3,5) -- (3.27,4.73);
    \draw[ultra thick] (4,4) -- (8,0);
    \draw[dashed] (3.27,4.73) -- (4,4);
    \draw[dotted] (3,5) -- (5,5);
    \node (tau) at (4,5.2) {$\tau$};
    \draw[dotted] (3,0) -- (3,5);
    \node (s) at (3.2,2) {$s$};
    \draw[dotted] (0,0) -- (3,0);
    \node (a) at (1.5,0.2) {$s - \tau$};
    \draw[thick,domain=-180:-85] plot ({1.75*cos(\x)},{1.75+0.3*sin(\x)});
    \draw[dashed,domain=-85:0] plot ({1.75*cos(\x)},{1.75+0.3*sin(\x)});
    \draw[dashed,domain=0:180] plot ({1.75*cos(\x)},{1.75+0.3*sin(\x)});
    \draw[dashed,domain=-180:-105] plot ({3+1.75*cos(\x)},{3.25+0.3*sin(\x)});
    \draw[thick,domain=-105:0] plot ({3+1.75*cos(\x)},{3.25+0.3*sin(\x))});
    \draw[dashed,domain=0:180] plot ({3+1.75*cos(\x)},{3.25+0.3*sin(\x))});
    \draw[thick] (0,5) ellipse (5 and 0.35);
    \draw[thick] (-2,0) arc(-180:0:5 and 0.35);
    \draw[dashed] (-2,0) arc(-180:-360:5 and 0.35);
    \draw[color=blue,domain=0:360] plot
    ({0.175*cos(\x)},{0.03*sin(\x))});
    \draw[color=blue,domain=0:360] plot
    ({3+0.175*cos(\x)},{5+0.03*sin(\x))});
    
    \draw[color=red,rotate around={32:(1.5,2.5)}] (-1.36,2.5) arc(-180:0:2.88 and 0.2);
    \draw[dashed,color=red,rotate around={32:(1.5,2.5)}] (-1.36,2.5) arc(-180:-360:2.88 and 0.2);
    \end{tikzpicture}
    \caption{One possible configuration of light cones for the unknown $\phi$ and the auxiliary multiplier $\psi$. The upward opening cone corresponds to $\phi$, while the downward opening one corresponds to $\psi$. The data for $\phi$ are supported near the tip of the upward opening cone, and an analogous statement holds for the data for $\psi$. We can think of the data as being supported in the blue circles around these tips. Because of this, the cones give good spacetime representations of the wave fronts associated to $\phi$ and $\psi$. We note that this is a $2 + 1$ dimensional model of the $3 + 1$ dimensional setting. The sections of the cones and the bounds for the support of the data, which appear as circles in the figure, are really spheres, and the red ellipse of intersection is really an ellipsoid. We also refer the reader to Figure~\ref{fig:rholowerbound} in Section~\ref{sec:geometry} to see what this looks like in each $\Sigma_t$. It is important to note that, when $\tau$ is very small relative to $s$, the two cones will have roughly parallel tangent planes near their points of intersection when $t \approx {s \over 2}$.}
    \label{fig:lightconesaux}
\end{figure}

For example, restricting to $\psi_0 = 0$ but allowing $\psi_1$ to vary in the class described above, we can use the first estimate in Proposition~\ref{prop:bilinen}, giving us that
\begin{equation}
    \begin{aligned}
        |\partial_t \phi| (s,x_0^i) \le C \sup_{|\alpha| \le 10} \sup_{\psi_1} \int_{B_1 (p)} \partial_t \psi \partial_t \partial^\alpha \phi d x = C \sup_{|\alpha| \le 10} \sup_{\psi_1} \int_{\Sigma_s} \partial_t \psi \partial_t \partial^\alpha \phi d x
        \\ \le C \sup_{|\alpha| \le 10} \sup_{\psi_1} \left [ \int_{\Sigma_0} (\partial_t \psi) (\partial_t \partial^\alpha \phi) + (\partial^i \psi) (\partial_i \partial^\alpha \phi) d x + C \int_0^s \int_{\Sigma_t} |\partial^\alpha (\Box \phi)| |\partial \psi| d x d t \right ].
    \end{aligned}
\end{equation}
We note that the terms involving spatial derivatives of $\psi$ are $0$ on $\Sigma_s$ because we are taking $\psi_0 = 0$. Now, the data for $\psi$ are compactly supported and are sufficiently regular so that $\partial \psi$ decays appropriately. Moreover, the data for $\phi$ are assumed to be supported in the unit ball $B$ centered at the origin in $\Sigma_0$. These two observations tell us that
\begin{equation}
    \begin{aligned}
        \sup_{|\alpha| \le 10} \sup_{\psi_1} \int_{\Sigma_0} (\partial_t \psi) (\partial_t \partial^\alpha \phi) + (\partial^i \psi) (\partial_i \partial^\alpha \phi) d x = \sup_{|\alpha| \le 10} \sup_{\psi_1} \int_B (\partial_t \psi) (\partial_t \partial^\alpha \phi) + (\partial^i \psi) (\partial_i \partial^\alpha \phi) d x
        \\ \le {C \over (1 + s)} \chi_{|\tau| \le 5} \Vert \partial \phi \Vert_{H^{10} (\Sigma_0)},
    \end{aligned}
\end{equation}
see Figure~\ref{fig:lightconesaux} (adapted from \cite{And21}) for a rough picture of the setting. We note that we are using the strong Huygens principle for $\psi$ in this case, giving us the cutoff function in $\tau$. While this simplifies matters, we do not believe this is necessary to make the argument work. Analogous arguments (but using the energies associated to $\partial_i$ instead of $\partial_t$ from Proposition~\ref{prop:bilinen}) give pointwise decay for the other unit partial derivatives of $\phi$. We record these results together in the following proposition.

\begin{proposition} \label{prop:badderpointwise}
Let
\[
M(s,x_0,y_0,z_0) := {1 \over (1 + s)} \chi_{|\tau| \le 5} \Vert \partial \phi \Vert_{H^{10} (\Sigma_0)} + \sup_{|\alpha| \le 10} \sup_{\psi_1} \int_0^s \int_{\Sigma_t} |\partial^\alpha (\Box \phi)| |\partial \psi| d x d t.
\]
Then, we have that
\[
|\partial \phi| (s,x_0,y_0,z_0) \le C M(s,x_0,y_0,z_0).
\]
\end{proposition}

\subsection{Improved decay of good derivatives} \label{sec:goodder}
We now use the second and third identities in Proposition~\ref{prop:bilinen} to get improved estimates on good derivatives of $\phi$. As was described in Section~\ref{sec:proofov}, the main idea is to exploit the fact that the good derivatives associated to a wave asymptotically look like translation vector fields in small regions. More precisely, we may assume without loss of generality that we aim to show improved decay for the good derivatives $\partial_v$, ${1 \over r} \Omega_{i j}$ for $\phi$ at some point $p$ lying on the positive $x$ axis in $\Sigma_s$. This point then has coordinates $(s,x_0,0,0)$. We recall that we denote the $u$ coordinate of this point by $\tau$, that is,
\[
\tau = s - x_0.
\]
At such a point $p$, we have that $L = \partial_t + \partial_r = \partial_t + \partial_x$ while ${1 \over r} \Omega_{x y} = \partial_y$ and ${1 \over r} \Omega_{x z} = \partial_z$ (the other usual rotation vector field vanishes at this point). This can be seen in Figure~\ref{fig:lightconesaux}. We thus use the second and third estimates from Proposition~\ref{prop:bilinen}. More precisely, for the second estimate, we use $\partial_i \in \{ \partial_y, \partial_z \}$ and for the third estimate, we use $\partial_i = \partial_x$. We take first the case of $\partial_y$. Proceeding as in Section~\ref{sec:badder} (but being a bit more careful with the multiplier), this gives us that
\begin{equation}
    \begin{aligned}
        |\partial_y \phi| (s,x_0^i) \le C \sup_{|\alpha| \le 10} \sup_{\psi_1} \int_{B_1 (p)} \partial_t \psi \partial_y \partial^\alpha \phi d x = C \sup_{|\alpha| \le 10} \sup_{\psi_1} \int_{\Sigma_s} \partial_t \psi \partial_y \partial^\alpha \phi d x
        \\ \le C \sup_{|\alpha| \le 10} \sup_{\psi_1} \left [ \int_{\Sigma_0} (\partial_t \psi) (\partial_y \partial^\alpha \phi) + (\partial_y \psi) (\partial_t \partial^\alpha \phi) d x + C \int_0^s \int_{\Sigma_t} |\partial^\alpha(\Box \phi)| |\partial_y \psi| d x d t \right ].
    \end{aligned}
\end{equation}
Examining the initial data term, we expect that $\partial_y$ is roughly a good derivative of $\psi$, and should thus be of size $1 / (1 + s)^2$, an improvement of $1 / (1 + s)$ over generic derivatives of $\psi$. Thus, we integrate one of the terms by parts and use the fact that $\phi$ is compactly supported, giving us that
\begin{equation}
    \begin{aligned}
        \sup_{|\alpha| \le 10} \sup_{\psi_1} \int_{\Sigma_0} (\partial_t \psi) (\partial_y \partial^\alpha \phi) + (\partial_y \psi) (\partial_t \partial^\alpha \phi) d x = \sup_{|\alpha| \le 10} \sup_{\psi_1} \int_{\Sigma_0} -(\partial_t \partial_y \psi) (\partial^\alpha \phi) + (\partial_y \psi) (\partial_t \partial^\alpha \phi) d x.
    \end{aligned}
\end{equation}
Now, we use Proposition~\ref{prop:planeder} below to decompose $\partial_y$ relative to the null frame adapted to $\psi$, and we see that, in the support of $\phi$, we have that either $\partial_y \psi = \overline{\partial}' \psi + O\left ((1 + s)^{-1} \right )\partial\psi$ or it is identically $0$ (by the strong Huygens principle), meaning that $|\partial_y \psi| \le C (1 + s)^{-2} \chi_{|\tau| \le 5}$, as was expected. Thus, we have that
\begin{equation}
    \begin{aligned}
        \sup_{|\alpha| \le 10} \sup_{\psi_1} \int_{\Sigma_0} (\partial_t \psi) (\partial_y \partial^\alpha \phi) + (\partial_y \psi) (\partial_t \partial^\alpha \phi) d x \le {C \over (1 + s)^2} \chi_{|\tau| \le 5} \Vert \phi \Vert_{H^{11} (\Sigma_0)}.
    \end{aligned}
\end{equation}
This analysis is completely analogous when replacing $\partial_y$ with $\partial_z$, and is only slightly different when considering $\partial_t + \partial_x$. Indeed, when considering $\partial_t + \partial_x$, the inequality we arrive at is
\begin{equation}
    \begin{aligned}
        |\partial_t \phi + \partial_x \phi| (s,x_0^i) \le C \sup_{|\alpha| \le 10} \sup_{\psi_1} \int_{B_1 (p)} \partial_t \psi \partial_y \partial^\alpha \phi d x = C \sup_{|\alpha| \le 10} \sup_{\psi_1} \int_{\Sigma_s} \partial_t \psi (\partial_t \partial^\alpha \phi + \partial_x \partial^\alpha \phi) d x
        \\ \le C \sup_{|\alpha| \le 10}\sup_{\psi_1} \int_{\Sigma_0} (\partial_t \psi + \partial_x \psi) (\partial_t \partial^\alpha \phi + \partial_x \partial^\alpha \phi) + (\partial_y \psi) (\partial_y \partial^\alpha \phi) + (\partial_z \psi) (\partial_z \partial^\alpha \phi) d x
        \\ + C \sup_{|\alpha| \le 10} \sup_{\psi_1} \int_0^s \int_{\Sigma_t} |\partial^\alpha(\Box \phi)| |(\partial_t + \partial_x) \psi| d x d t.
    \end{aligned}
\end{equation}
Thus, when examining the integral on $\Sigma_0$, there is no longer a need to integrate by parts, and we can directly use Proposition~\ref{prop:planeder} to say that the expressions involving $\psi$ are of size $C (1 + s)^{-2} \chi_{|\tau| \le 5}$ in the support of $\phi$. We now record these results.

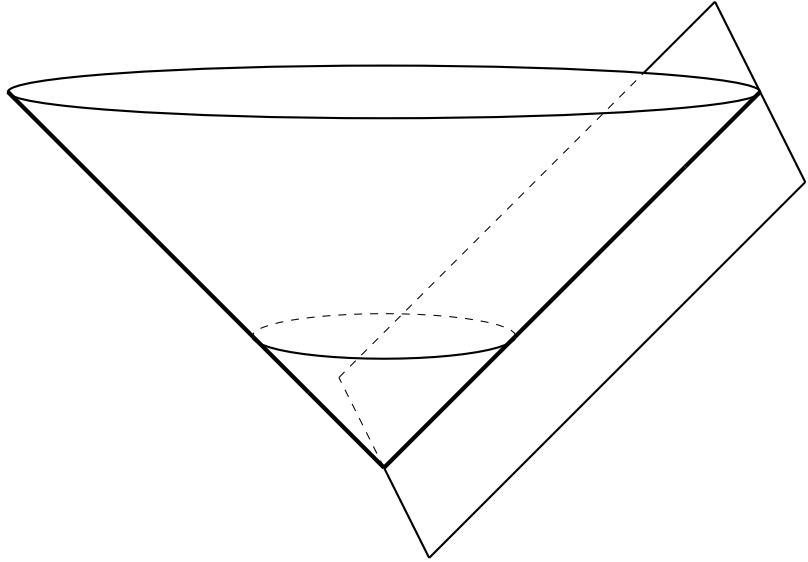
\begin{figure}
    \centering
    \begin{tikzpicture}
    \draw[ultra thick] (0,0) -- (4,4);
    \draw[ultra thick] (4,4) -- (5,5);
    \draw[ultra thick] (0,0) -- (-1,1);
    \draw[ultra thick] (-1,1) -- (-5,5);
    \draw[thick,domain=-180:0] plot ({1.75*cos(\x)},{1.75+0.3*sin(\x)});
    \draw[dashed,domain=0:180] plot ({1.75*cos(\x)},{1.75+0.3*sin(\x)});
    \draw[thick] (0,5) ellipse (5 and 0.35);
    \draw[thick] (5.6,3.8) -- (4.4,6.2);
    \draw[thick] (0.6,-1.2) -- (0,0);
    \draw[thick] (0.6,-1.2) -- (5.6,3.8);
    \draw[dashed] (0,0) -- (-0.6,1.2);
    \draw[dashed] (-0.6,1.2) -- (3.45,5.25);
    \draw[thick] (3.45,5.25) -- (4.4,6.2);
    
    \end{tikzpicture}
    \caption{One possible configuration of the light cone for the unknown $\phi$ and an approximating flat null plane. For $t$ sufficiently large, the curvatures of the spherical sections decay, and the null plane becomes a better and better approximation. This is also true for the null plane and a light cone adapted to an auxiliary multiplier, although the approximation becomes better for sufficiently large $s - t$ in that case. This is formalized in Proposition~\ref{prop:planeder}, which shows that the tangent spaces of light cones becomes aligned with the tangent space of approximating null planes. This also helps explain the behavior of the angle between the light cones of the solution and auxiliary light cones. For $t \approx {s \over 2}$, both light cones are well approximated by the null plane, meaning that the angle between all three null hypersurfaces is small.}
    \label{fig:lightconesplane}
\end{figure}

\begin{proposition} \label{prop:goodderpointwise}
We have the following quantities which control the good derivatives of $\phi$.
\begin{enumerate}
    \item Let
    \[
    M_y (s,x_0) := {1 \over (1 + s)^2} \chi_{|\tau| \le 5} \Vert \phi \Vert_{H^{11} (\Sigma_0)} + \sup_{|\alpha| \le 10} \sup_{\psi_1} \int_0^s \int_{\Sigma_t} |\partial^\alpha (\Box \phi)| |\partial_y \psi| d x d t.
    \]
    Then, we have that
    \[
    |e_{x y} \phi| (s,x_0,0,0) = |\partial_y \phi| (s,x_0,0,0) \le C M_y (s,x_0).
    \]
    \item Let
    \[
    M_z (s,x_0) := {1 \over (1 + s)^2} \chi_{|\tau| \le 5} \Vert \phi \Vert_{H^{11} (\Sigma_0)} + \sup_{|\alpha| \le 10} \sup_{\psi_1} \int_0^s \int_{\Sigma_t} |\partial^\alpha (\Box \phi)| |\partial_z \psi| d x d t.
    \]
    Then, we have that
    \[
    |e_{x z} \phi| (s,x_0,0,0) = |\partial_z \phi| (s,x_0,0,0) \le C M_z (s,x_0).
    \]
    \item Let
    \[
    M_v (s,x_0) := {1 \over (1 + s)^2} \chi_{|\tau| \le 5} \Vert \partial \phi \Vert_{H^{11} (\Sigma_0)} + \sup_{|\alpha| \le 10} \sup_{\psi_1} \int_0^s \int_{\Sigma_t} |\partial^\alpha (\Box \phi)| |(\partial_t + \partial_x) \psi| d x d t.
    \]
    Then, we have that
    \[
    |L \phi| (s,x_0,0,0) = |(\partial_t + \partial_x) \phi| (s,x_0,0,0) \le C M_v (s,x_0).
    \]
    Thus, all taken together,we  have that
    \[
    |\overline{\partial} \phi| (s,x_0,0,0) \le C \left [M_y (s,x_0) + M_z (s,x_0) + M_v (s,x_0) \right ].
    \]
\end{enumerate}
\end{proposition}
This Proposition shall be used in the proof of Theorem~\ref{thm:main} to show improved decay of good derivatives. Indeed, there is a clear generalization to points with nonzero $y$ and $z$ coordinates.

Before proceeding, we make a few brief remarks about Proposition~\ref{prop:goodderpointwise} and Proposition~\ref{prop:badderpointwise}. We first note that the cutoff in $\tau$ comes from using the strong Huygens principle for $\psi$. While this leads to technical simplifications, we do not think that this is necessary for the proof of Theorem~\ref{thm:main}. If we did not have the strong Huygens principle, the pointwise estimates would have polynomial tails in $\tau$ instead of a cutoff function. The main difference is that the error integrals would require more thought to control.

We also note that there is a slight asymmetry between $M_v$ and $M_y$ (and, thus, also $M_z$). Indeed, the $L^2$ norm for $\phi$ appears in $M_y$ and $M_z$, but not $M_v$. This is because we had to perform an integration by parts on the data to get good derivatives on $\psi$. A more careful integration by parts argument can make this a bit more symmetric (we can decompose $\partial_t$ relative to the frame consisting of $\overline{\partial}'$ and $\partial_{r'}$, and integrate the $\partial_{r'}$ term by parts again after putting the good derivative on $\psi$). However, this is not necessary for the proof of Theorem~\ref{thm:main}.

There is a related way of thinking about establishing improved decay of good derivatives. We may take a weighted commutator, $\Gamma$, and commute it through the equation, and then try to show regular pointwise decay for this quantity using Proposition~\ref{prop:badderpointwise}. The algebraic identities \eqref{eq:GoodDerFrameDecomp} will then establish improved pointwise decay for good derivatives. This requires controlling integrals of the form
\[
\int \int \Gamma(f m(d \phi,d \phi)) \partial_t \psi d x d t.
\]
We may then intgrate the $\Gamma$ by parts so that it falls on $\partial_t \psi$, and then decompose $\Gamma$ relative to a null frame adapted to $\psi$. Doing this requires the geometric considerations described in Section~\ref{sec:geometry}.

Finally, we note that the most difficult part of using the above results to prove pointwise decay comes from controlling the error integrals because we must show that they are appropriately small in $s$ and $\tau$. Moreover, expect that the error integrals for $M_y$, $M_z$, and $M_v$ are better than the error integrals for $M$. Indeed, Proposition~\ref{prop:planeder} shows that the multipliers are effectively good derivatives of $\psi$ in a large region. This is the same mechanism that gave us improvements on the integrals over $\Sigma_0$ for $M_y$, $M_z$, and $M_v$.

\section{Geometric estimates on cones} \label{sec:geometry}
In this Section, we shall establish the necessary geometric estimates involving the solution cone and an auxiliary cone. The solution cones are parameterized by $u = t - r$. As in Section~\ref{sec:notation}, without loss of generality, we assume that we are taking an auxiliary multiplier with data given in a ball around some point $p$ with coordinates $(s,a,0,0)$ where $a>0$. We recall that $\tau = s - a$ is the $u$ coordinate of this point. The auxiliary cones are then parameterized by $u' = t - r'$. This section is mainly about proving estimates on the geometry of the intersections of cones of constant $u$ and constant $u'$.

With this setting at hand, we can describe the different estimates we shall need. The first result (Proposition~\ref{prop:rr'coordinates}, which is essentially taken from \cite{AndPas22}) describes coordinate systems which are adapted to both cones. This is useful because the error integrals will involve functions with decay properties adapted to these cones. Lemma~\ref{lem:rtr's-t} consists of preliminary geometric results which will be useful when analyzing the geometry more carefully in the rest of this Section. Then, the next two results (Lemma~\ref{lem:rhorr'} and Lemma~\ref{lem:rholowerbound}) establish bounds on $\rho$ in terms of $r$, $r'$, and $\tau$ (see Section~\ref{sec:notation}).

The final two results concern changes of frames. We can write the bad derivative of the solution $\underline{L}$ in terms of $\overline{\partial}$ and $\overline{\partial}'$ with coefficients which degenerate at a rate coming from the angle between the cones of constant $u$ and of constant $u'$. This is done explicitly in Proposition~\ref{prop:Lbardbardbar'}. Similarly, we can write $\hat{\partial}$ derivatives in terms of $\overline{\partial}'$ derivatives and $\underline{L}'$, where the $\underline{L}'$ coefficient is comparable to the angle between the cone of constant $u'$ and the null hyperplane of constant $\hat{u}$. This is done explicitly in Proposition~\ref{prop:planeder}. Thus, the small angles lead to worse estimates in Proposition~\ref{prop:Lbardbardbar'}, but better estimates in Proposition~\ref{prop:planeder}.

Before proceeding to the proof, we note that geometrically, one of the most important things to consider is the background Euclidean angle between the various null cones and null hyperplanes. In fact, most of the results below could be written in a more general fashion in terms of the angle between different hypersurfaces. The behavior of this angle is what determines the size of factors in changes of frames and volume forms, and estimating these factors is essentially the content of this Section. Even though we will not explicitly use the angles in our proofs, the intuition that comes from considering this will be very helpful. We also note that the invariant (i.e., not relying on the background Euclidean structure) objects behind the angle just described are various change of frame transformations. The angle is then essentially the Jacobian (or the inverse of the Jacobian depending on which direction we are changing frame) of this transformation.

We now state and prove the first result. We recall that $r$ is the usual radial coordinate while $r'$ is the radial coordinate adapted to $(s,a,0,0)$. In a fixed $\Sigma_t$, the level sets of $r$ and $r'$ intersect in a circle, and we also recall that $\theta$ denotes the usual angular coordinate for this circle (see Section~\ref{sec:notation}).

\begin{proposition} \label{prop:rr'coordinates}
We have the following expressions for the volume form in various coordinate systems.
\begin{enumerate}
    \item In $(r,r',\theta)$ coordinates, we have that
    \[
    d x \wedge d y \wedge d z = {r r' \over a} d r \wedge d r' \wedge d \theta.
    \]
    \item In $(t,u,u',\theta)$ coordinates, we have that
    \[
    d t \wedge d x \wedge d y \wedge d z = {r r' \over 2 a} d t \wedge d u \wedge d u' \wedge d \theta.
    \]
\end{enumerate}
\end{proposition}
\begin{proof}
The proof of these statements follows in the same way as the proof of Lemma 5.2 in \cite{AndPas22}.
\end{proof}

We now get bounds on $\rho$ in terms of $\tau$ and $r$. This will be used in other places where we need to control the geometry. We shall first require an expression for $\rho$ as a function of $\tau$, $r$, $r'$, and $s$ (we recall that $a = s - \tau$).

\begin{lemma} \label{lem:rtr's-t}
Let $\tau \ge 100$ and let $|u'| \le 5$ and $u\ge -5$.
\begin{enumerate}
    \item We have that $r + r' - a \le 2 r$, and similarly that $r + r' - a \le 2 r'$.
    \item We have that $r \ge {\tau - u - u' \over 2} \ge 0$, and similarly that $r' \ge {\tau - u - u' \over 2} \ge 0$.
    \item We have that $t \ge {\tau \over 10}$. Moreover, in the region where $|u| \le \delta \tau$, we have that $r \ge {99 \over 100} t$.
    \item In the region where $u \le \delta \tau$ and $r' \le s - t + 1$, we have that $s - t \ge {\tau \over 10}$.
    \item We have that $2 + r' \ge {99 \over 100} (1 + s - t)$. Moreover, assuming that $|u| \le \delta \tau$, we have that ${99 \over 100} (1 + s - t) \ge {\tau \over 10}$.
\end{enumerate}
\end{lemma}
\begin{proof}
This is adapted from Lemma 12 in \cite{And21}, and the proof follows in a similar way.
\end{proof}

\begin{lemma} \label{lem:rhorr'}
Under the same conditions as Lemma~\ref{lem:rtr's-t}, the following are true.
\begin{enumerate}
\item We have that
\begin{equation}
    \begin{aligned}
    \rho^2 = r^2 - r^2 \left (1 + {\tau - u - u' \over a} - {\tau - u - u' \over r} - {(\tau - u - u')^2 \over 2 r a} \right )^2 \le C r (\tau - u - u').
    \end{aligned}
\end{equation}
\item We have that
\begin{equation}
    \begin{aligned}
    \rho^2 = (r')^2 - (r')^2 \left (1 + {\tau - u - u' \over a} - {\tau - u - u' \over r'} - {(\tau - u - u')^2 \over 2 r' a} \right )^2 \le C r' (\tau - u - u').
    \end{aligned}
\end{equation}
\end{enumerate}
Moreover, assuming that $\tau \le \alpha s$ for some $0 < \alpha < 1$, we have that $a \ge (1 - \alpha) s$. Similarly, assuming that $\tau \ge \alpha s$ for some $0 < \alpha < 1$, we have that $a \le (1 - \alpha) s$.
\end{lemma}
\begin{proof}
The first two parts follow as in the proof of Lemma 14 in \cite{And21}. The final two statements follow immediately from noting that $s = a + \tau$.
\end{proof}

\begin{lemma} \label{lem:rholowerbound}
When $r \le r'$ and $|\pi - \vartheta| \ge \delta$, we have that
\[
\rho^2 \ge c r (\tau - u - u').
\]
Similarly, when instead $r' \le r$ and $|\pi - \vartheta'| \ge \delta$, we have that
\[
\rho^2 \ge c r' (\tau - u - u').
\]
In both statements, the constant $c$ depends only on $\delta$.
\end{lemma}
\begin{proof}
We shall prove the statement when $r \le r'$ and $|\pi - \vartheta| \ge \delta$. The analogous statement when $r' \le r$ and $|\pi - \vartheta'| \ge \delta$ follows in the same way by symmetry.

Given that $r \le r'$, we now consider two regions, one where $|\vartheta| \le \delta$, and the other consisting of the remaining points, i.e., the points where $\delta \le \vartheta \le \pi - \delta$ and where $\pi + \delta \le \vartheta \le 2 \pi - \delta$. In this second region of points, we note that $\rho = r \sin(\vartheta)$ and $r$ are comparable, meaning that we have that
\[
\rho^2 \ge 2c r^2 \ge c r (\tau - u - u'),
\]
as desired. Thus, we must only consider the case where $|\vartheta| \le \delta$.

\begin{figure}
    \centering
    \begin{tikzpicture}
    \draw[dashed] (0,0) -- (7,0);
    \draw[dashed] (3.26,0.5) -- (3.26,2.35);
    \draw[very thick] (0,0) circle (4);
    \draw[thick] (2.6,0) -- (4,0);
    \filldraw (3.26,0) circle (0.05);
    \node (d) at (3.26,-0.25) {$d$};
    \node (q) at (3.26,0.25) {$(q,0)$};
    \draw[very thick] (7,0) circle (4.4);
    \draw[dashed] (0,0) -- (3.3,2.35);
    \node (r) at (1.55,1.275) {$r$};
    \draw[dashed] (7,0) -- (3.3,2.30);
    \node (r') at (5.4,1.275) {$r'$};
    \node (theta) at (0.75,0.23) {$\vartheta$};
    \node (theta') at (5.8,0.23) {$\vartheta'$};
    \end{tikzpicture}
    \caption{The geometry involved in Lemma~\ref{lem:rholowerbound} when $r \le r'$. Because $r \le r'$, the distance between $(q,0)$ and the circle on the right is at least ${d \over 2}$, while the distance between $(q,0)$ and the circle on the left is at most ${d \over 2}$.}
    \label{fig:rholowerbound}
\end{figure}
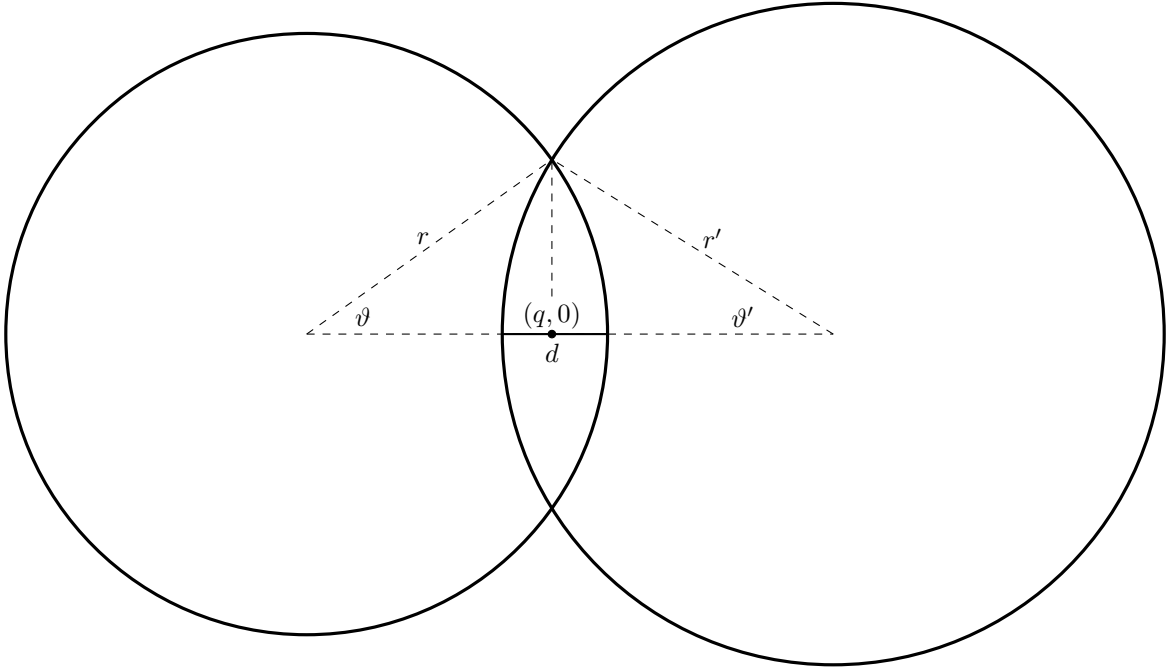

In this region, we look in the $2$ dimensional plane formed by intersecting $\Sigma_t$ and the plane $z = 0$ (we do this because, using symmetry, we can assume that $z = 0$ and $y = \rho$ without loss of generality). We have that $r = t - u$ and $r' = s - t - u'$. Moreover, we have that $r + r' = s - \tau + d$, where $d$ denotes the distance between the points on the $x$ axis where $x = r$ and where $a-x = r'$ (see Figure~\ref{fig:rholowerbound}). Thus, we have that $d = \tau - u - u'$. If we denote by $p$ the point with positive $y$ coordinate where the circles of constant $r$ and $r'$ intersect in the plane we are considering, we look at the right triangle formed by dropping a perpendicular from $p$ to the $x$ axis. If we denote by $q$ the $x$ coordinate of the point where this perpendicular intersects the $x$ axis, we note that $r - q \ge {d \over 2} = {\tau - u - u' \over 2}$ because $r \le r'$, meaning that the circle of radius $r$ has larger curvature than the circle of radius $r'$. From this, if we denote the $x$ and $y$ coordinates of $p$ by $x(p)$ and $y(p)$, respectively, we have that
\[
y^2 (p) = r^2 - x^2 (p) = (r + x(p)) (r - x(p)) \ge {1 \over 2} (r + x(p)) (\tau - u - u'),
\]
meaning that
\[
\rho^2 = y^2 \ge {1 \over 2} r (\tau - u - u'),
\]
as desired.
\end{proof}
We now compute the following changes of frame.

\begin{proposition} \label{prop:Lbardbardbar'}
Let $(1 + 10 \delta) \tau \ge \tau - u - u' > 0$ and $u>-5, |u'|\le 5, s\ge 10$. We have that
\[
\underline{L} = \gamma_1 \overline{\partial} + \gamma_2 \overline{\partial}'
\]
for coefficients $\gamma_1$ and $\gamma_2$ with the following properties:
\begin{enumerate}
    \item If $\tau \le {1 \over 2} s$, then
    \[
    |\gamma_1| \le {C \min(\sqrt{r},\sqrt{r'}) \over \sqrt{\tau - u - u'}} \hspace{5 mm} \text{and} \hspace{5 mm} |\gamma_2| \le {C \min(\sqrt{r},\sqrt{r'}) \over \sqrt{\tau - u - u'}}.
    \]
    \item If $\tau \ge \delta s$, and additionally $|u| \le {1 \over 10} \tau$ and $|u'| \le {1 \over 10} \tau$, then
    \[
    |\gamma_1| \le C \hspace{5 mm} \text{and} \hspace{5 mm} |\gamma_2| \le C.
    \]
\end{enumerate}
\end{proposition}
\begin{proof}
We begin by describing the situation geometrically. The enemy is when the tangent planes of constant $u$ and constant $u'$ hypersurfaces coincide, meaning that the angle between the hypersurfaces is very close to $0$. Geometrically, it seems as though this can only happen when the following three conditions are met: $\tau$ is very small relative to $s$ (note the worst estimate when $\tau \le {1 \over 2} s$), $\vartheta$ is very close to $0$, and $\vartheta'$ is very close to $0$. This will be made more apparent in the course of the proof.

Without loss of generality, we can assume that the auxiliary cones are adapted to the point $(s,a,0,0)$ with $a>0$ (see the discussion in Section~\ref{sec:notation}). Given any point in the region in question, we can moreover assume that $z = 0$ and $y \ge 0$. Thus, we can assume that $\rho = y$.

Turning first to the case of $\tau \le {1 \over 2} s$ and recalling the geometric quantities $\vartheta$ and $\vartheta'$ (see Section~\ref{sec:notation} and also Figure~\ref{fig:rholowerbound}), we now consider various regions depending on these parameters. The first region is where $|\vartheta - \pi| \ge \delta$ and $|\vartheta' - \pi| \ge \delta$. At these points, we shall write $\partial_r$ in terms of angular vector fields adapted to both centers. The bad coefficients will essentially come from the angle between the two cones going to $0$ when both $\vartheta$ and $\vartheta'$ approach $0$. It is better to use the rotation vector fields here because the vector fields $L$ and $L'$ form smaller angles. The two remaining cases are when one of $|\vartheta - \pi| \le \delta$ or $|\vartheta' - \pi| \le \delta$. At these points, the cones are actually uniformly transverse to each other, but we must use $L$ and $L'$ instead of rotations.

We now consider a point in the region where $|\vartheta - \pi| \ge \delta$ and $|\vartheta' - \pi| \ge \delta$. We recall that
\[
\partial_r = {x \over r} \partial_x + {y \over r} \partial_y.
\]
Moreover, we have that
\[
e_{x y} = {1 \over r} \Omega_{x y} = {x \over r} \partial_y - {y \over r} \partial_x,
\]
and that
\[
e_{x y}' = {1 \over r'} \Omega_{x' y'} = {x' \over r'} \partial_y - {y' \over r'} \partial_x = {x - a \over r'} \partial_y - {y \over r'} \partial_x.
\]
We shall write $\partial_x$ and $\partial_y$ in terms of $e_{x y}$ and $e_{x y}'$. We have that
\[
{r \over r'} e_{x y} - e_{x y}' = {x \over r'} \partial_y - {x - a \over r'} \partial_y = {a \over r'} \partial_y,
\]
meaning that
\[
\partial_y = {r \over a} e_{x y} - {r' \over a} e_{x y}'.
\]
Moreover, we have that
\[
{r (x - a) \over x r'} e_{x y} - e_{x y}' = {y \over r'} \partial_x - {y (x - a) \over r' x} \partial_x = {1 \over x r'} (y x - y (x - a)) \partial_x = {y a \over x r'} \partial_x,
\]
meaning that
\[
\partial_x = {r (x - a) \over y a} e_{x y} - {x r' \over y a} e_{x y}'.
\]
From this, it follows that
\[
\partial_r = {x (x - a) \over y a} e_{x y} - {x^2 r' \over a y r} e_{x y}' + {y \over a} e_{x y} - {y r' \over a r} e_{x y}' = {x (x - a) + y^2 \over y a} e_{x y} - {r r' \over a y} e_{x y}'.
\]
Now, we note that $\underline{L} = L - 2 \partial_r$. Thus, we have that
\[
\underline{L} = L - {2 x (x - a) + 2y^2 \over y a} e_{x y} - {2 r r' \over a y} e_{x y}'.
\]
Now, because $\tau \le {1 \over 2} s$, we note that $a \ge {s \over 2}$. Because $r \le s + 10$ and $r' \le s + 10$, we have that
\[
\left |{2 x (x - a) + 2y^2 \over y a} \right | \le {C \min(r,r') \over y} = {C \min(r,r') \over \rho},
\]
and similarly that
\[
\left |{2 r r' \over a y} \right | \le {C \min(r,r') \over y} = {C \min(r,r') \over \rho}.
\]
The desired result follows from the fact that
\[
{\min(r,r') \over \rho} \le C {\min(\sqrt{r},\sqrt{r'}) \over \sqrt{\tau - u - u'}}
\]
when $|\vartheta - \pi| \ge \delta$ and $|\vartheta' - \pi| \ge \delta$, which follows from Lemma~\ref{lem:rholowerbound}.

We now assume that $|\vartheta - \pi| \le \delta$. We recall that
\[
L = \partial_t + \partial_r = \partial_t + {x \over r} \partial_x + {y \over r} \partial_y,
\]
and that
\[
L' = -\partial_t + \partial_{r'} = -\partial_t + {x - a \over r'} \partial_x + {y \over r'} \partial_y.
\]
We also recall that
\[
\underline{L} = \partial_t - \partial_r = \partial_t - {x \over r} \partial_x - {y \over r} \partial_y,
\]
and that
\[
\underline{L}' = -\partial_t - \partial_{r'} = -\partial_t - {x - a \over r'} \partial_x - {y \over r'}\partial_y.
\]
Geometrically, it seems as though $\underline{L} \approx -L'$ in this region. Thus, we subtract this off, and look at the difference. We have that 
\[
\underline{L} + L' = \partial_{r'} - \partial_r = \left [{x - a \over r'} - {x \over r} \right ] \partial_x +\left [\frac y {r'}-\frac yr\right ] \partial_y
\]
Moreover, we have that
\[
\partial_x = {x \over r} \partial_r - {y \over r} e_{x y},
\]

and we have that
\[
\partial_y = {y \over r} \partial_r + {x \over r} e_{x y}.
\]
Thus, we have that
\[
\underline{L} + L' = {r (x - a) - x r' \over r r'} \left [ {x \over 2 r} (L - \underline{L}) - {y \over r} e_{x y} \right ] + {r y - r' y \over r r'} \left [{y \over 2 r} (L - \underline{L}) + {x \over r} e_{x y} \right ].
\]
From this, it follows that
\begin{equation} \label{eq:LbarL'}
    \begin{aligned}
    \left [{x r (x - a) - x^2 r' \over 2 r^2 r'} + {r' y^2 - r y^2 \over 2 r^2 r'} + 1 \right ] \underline{L}
    \\ = {r (x - a) - x r' \over r r'} \left [ {x \over 2 r} L - {y \over r} e_{x y} \right ] - L' + {r y - r' y \over r r'} \left [{y \over 2 r} L + {x \over r} e_{x y} \right ].
    \end{aligned}
\end{equation}
The coefficients on the right hand side are uniformly bounded. We must now only show that the coefficient in front of $\underline{L}$ on the left hand side is not close to $0$.

We thus examine the expression
\[
{x r (x - a) - x^2 r' \over 2 r^2 r'} + {r' y^2 - r y^2 \over 2 r^2 r'} = \left [{x - a \over r'} - {x \over r} \right ] {x \over 2 r} + {r' y^2 - r y^2 \over 2 r^2 r'}.
\]
Geometrically, we expect that both of these terms are very close to $0$. Beginning with the second term involving $y$, we simply use the fact that $|\pi - \vartheta| \le \delta$ and that $y = r \sin(\vartheta)$. Using this along with the fact that $r \le r'$ in this region gives us that
\[
\left |{r' y^2 - r y^2 \over 2 r^2 r'} \right | \le |\sin^2 (\vartheta)| \le \delta^2.
\]
Then, for the term involving $x$, we have that
\[
{x - a \over r'} - {x \over r} = -\cos(\vartheta') - \cos(\vartheta).
\]
Moreover, we note that $|\vartheta'| \le |\vartheta - \pi|$. This is because $r' > r$ and $y = r \sin(\vartheta) = r' \sin(\vartheta').$ Thus, taking a Taylor expansion gives us that
\[
{x - a \over r'} - {x \over r} = -\cos(\vartheta') - \cos(\vartheta) = O((\vartheta - \pi)^2) = O(\delta^2).
\]
Examining \eqref{eq:LbarL'}, this implies the desired result. When $|\vartheta' - \pi| \le \delta$ instead, a symmetric argument allows us to write $\underline{L}'$ in terms of $\overline{\partial}$ and $\overline{\partial}'$. The desired result in this region then follows because $\underline{L}$ may be expanded in the null frame adapted to the auxiliary cone with coefficients of size about $1$.

We now consider the case where $\tau \ge \delta s$. This follows immediately from considering the case of $s = 1$ and $\tau \ge \delta$. The cones are uniformly transverse in this range, meaning that the result is true. Rescaling back then gives us that the cones are still uniformly transverse because scaling is a conformal transformation. This implies the desired result.
\end{proof}

We now turn to frame conversions between rectangular derivatives and the null frame adapted to the auxiliary multiplier. This is needed in order to show that the error integrals that come up when applying Proposition~\ref{prop:goodderpointwise} are, in fact, better. Now, because good derivatives and generic derivatives behave the same way when $\tau$ is comparable to $s$, we only need to perform this frame conversion when, say, $\tau \le \delta s$.

\begin{proposition} \label{prop:planeder}
Let $\tau \le \delta s$, let $|u'| \le 5$, and let $|\vartheta'| \le {\pi \over 2}$, and let $h$ be an arbitrary smooth function. Moreover, we assume that we have chosen coordinates such that the data for the auxiliary multiplier are supported in a unit ball in $\Sigma_s$ with coordinates $(s,a,0,0)$, and where $\tau = s - a$ (see Section~\ref{sec:notation}).
\begin{enumerate}
    \item We have that
    \[
    |\partial_y h| \le C |\overline{\partial}' h| + C {\sqrt{\tau - u - u'} \min(\sqrt{1 + t},\sqrt{1 + s - t}) \over (1 + s - t)} |\partial h|.
    \]
    \item We have that
    \[
    |\partial_z h| \le C |\overline{\partial}' h| + C {\sqrt{\tau - u - u'} \min(\sqrt{1 + t},\sqrt{1 + s - t}) \over (1 + s - t)} |\partial h|.
    \]
    \item We have that
    \[
    |(\partial_t + \partial_x) h| \le C |\overline{\partial}' h| + C {(\tau - u - u') \min(1 + t,1 + s - t) \over (1 + s - t)^2} |\partial h|.
    \]
\end{enumerate}
Thus, we can generally conclude that
\[
|\hat{\partial} h| \le C |\overline{\partial}' h| + C {\sqrt{\tau - u - u'} \min(\sqrt{1 + t},\sqrt{1 + s - t}) \over (1 + s - t)} |\partial h|.
\]
\end{proposition}
\begin{proof}
We shall use the Euclidean inner product on $\R^4$ as an easy way to compute changes of frames.

The above computations amount to computing the projection of $\hat{\partial}$ in the $\underline{L}'$ direction. Because the $\hat{\partial}$ derivatives have Euclidean length comparable to $1$ and so do the vector fields in the null frame for $\psi$, the projections onto the good derivatives $\overline{\partial}'$ must have length comparable to $1$. We must show that the projections onto the bad derivative have small length.

We thus compute that
\[
\langle \partial_y,\underline{L}' \rangle_e = \langle \partial_y,{x - a \over r'} \partial_x + {y \over r'} \partial_y + {z \over r'} \partial_z \rangle_e = {y \over r'}.
\]
An application of Lemma~\ref{lem:rhorr'} then gives us the desired result for the $\partial_y$ derivative because $|y| \le \rho$, and an analogous argument works for $\partial_z$.

Now, we have that
\[
\langle \partial_t + \partial_x,\underline{L}' \rangle_e = \langle \partial_t + \partial_x,-\partial_t - {x - a \over r'} \partial_x + {y \over r'} \partial_y + {z \over r'} \partial_z \rangle_e = -1 + {a - x \over r'} = {a - x - r' \over r'}.
\]
Now, we have that $a - x = r' \cos(\vartheta')$, meaning that
\[
{a - x - r' \over r'} = \cos(\vartheta') - 1.
\]
Now, for $-{\pi \over 2} \le \vartheta' \le {\pi \over 2}$, we have that
\[
\cos(\vartheta') - 1 \le C (\vartheta')^2 \le C (\sin(\vartheta'))^2 = {C \rho^2 \over (r')^2}.
\]
An application of Lemma~\ref{lem:rhorr'} gives us the desired result in this region, and a similar argument works for ${\pi \over 2} \le \vartheta' \le {3 \pi \over 2}$ by comparing $\cos(\vartheta') - 1$ with $(\vartheta' - \pi)^2$.

\end{proof}

\section{Weighted estimates of Dafermos--Rodnianski} \label{sec:linests}
In this Section, we record the other linear estimates we shall use, coming from the work \cite{DafRod10} of Dafermos--Rodnianski. These estimates will all be in terms of general solutions to inhomogeneous wave equations. To that end, we consider the equation
\begin{equation} \label{eq:generalwave}
\begin{aligned}
    \Box \phi = F,
    \\ \phi(0,x) = \phi_0 (x),
    \\ \partial_t \phi (0,x) = \phi_1 (x).
\end{aligned}
\end{equation}
Moreover, we shall specialize slightly by assuming that $F$ is supported in the set where $u \ge -2$ or $v \ge -2$, and that $\phi_0$ and $\phi_1$ are supported in the unit ball in $\Sigma_0$. For convenience, we will assume $\phi$ is defined for $-20\le t\le T$.

The $r^p$ estimates are reviewed in Section~\ref{sec:r^p}. The decay in of energy fluxes, which is a consequence of the $r^p$ weighted energy estimates, is reviewed in Section~\ref{sec:uenest}. That the $r^p$ estimates lead to decay in energy fluxes was already present in the original work \cite{DafRod10} of Dafermos--Rodnianski, and the estimates in this Section are minor variations on the themes introduced therein.

\subsection{$r^p$ estimates} \label{sec:r^p}
We begin by recording the $r^p$ energy estimates for $0 \le p \le 2$, which are the main tool used for proving energy decay.
\begin{proposition} \label{prop:r^p}
Let $\phi$ be a solution to \eqref{eq:generalwave}, let $T > 0$, and let $u_1$ and $u_2$ be such that $u_1 \le u_2$. Then, we have the following estimates adapted to the truncated cones $C_{u_1}^T$ and $C_{u_2}^T$.
\begin{equation}
    \begin{aligned}
    \int_{u_2}^{2 T - u_2} \int_{S^2} r^p (L(r \phi))^2 d \omega d v + \int_{u_1}^{u_2} \int_u^{2 T - u} \int_{S^2} pr^{p - 1} (L(r \phi))^2 + (2 - p) r^{p - 1} |e_A (r \phi)|^2 d \omega d v d u
    \\ \le C \int_{u_1}^{u_2} \int_u^{2 T - u} \int_{S^2} |\Box \phi| r^{p + 1} |L (r \phi)| d \omega d v d u + C \int_{u_1}^{2 T - u_1} \int_{S^2} r^p (L(r \phi))^2 d \omega d v.
    \end{aligned}
\end{equation}
The limiting estimates also hold as $T \rightarrow \infty$.
\end{proposition}

\begin{proof}
The proof can be found in \cite{DafRod10}. To summarize, we can proceed by writing the equation for $\phi$ as
\[
-L \underline{L} (r \phi) + r^{-2} \slashed{\Delta}_0 (r \phi) = r \Box \phi,
\]
multiplying by $r^p L (r \phi)$, and integrating by parts in the spacetime region bounded by the truncated null cones $C_{u_1}^T$ and $C_{u_2}^T$, and the spacelike hypersurface $\Sigma_T$. We note that we are in fact dropping positive terms on the portion of $\Sigma_T$ between the cones $C_{u_2}$ and $C_{u_1}$ corresponding to the $r^p$ flux through this region.
\end{proof}

\subsection{Energy decay} \label{sec:uenest}
We now record decay of the usual $\partial_t$ energy flux that comes as a result of the $r^p$ estimates listed in Section~\ref{sec:r^p}. We recall that the main use of these weighted estimates is to show that the energy decays in $u$, showing us that the solution is better behaved on average than the pointwise estimates we shall prove (see Section~\ref{sec:proofov}). We shall show that the usual $\partial_t$ and $r^p$ with $p = 1$ energy fluxes decay in $u$ (Proposition~\ref{prop:uendec}), and that an averaged energy estimate decays in $u$ (Proposition~\ref{prop:avenestabovecone}). The fact that the energy flux decays in $u$ will allow us to show that energy fluxes through null cones adapted to auxiliary multipliers decay (Proposition~\ref{prop:enestabovecone}).

We now precisely state the energy decay statements in the form we shall use them.

\begin{proposition} \label{prop:uendec}
Let $\phi$ be a solution to \eqref{eq:generalwave}, let $T > 0$, and let $u_0 \ge 10$. Then, the following two estimates hold:
\begin{enumerate}
    \item The $r^p$ energy flux with $p = 1$ through the truncated outgoing cone $C_{u_0}^T$ decays in terms of appropriate norms. More precisely, we have that
    \begin{equation}
        \begin{aligned}
        \int_{u_0}^{2 T - u_0} \int_{S^2} r |L(r \phi)|^2 d \omega d v \le C \int_{u_0 / 10}^{u_0} \int_u^{2 T - u} \int_{S^2} |\Box \phi| r^2 |L(r \phi)| d \omega d v d u
        \\ + {C \over \langle u_0 \rangle} \left [\int_{-10}^{u_0} \int_u^{2 T - u} \int_{S^2} |\Box \phi| r^3 |L(r \phi)| d \omega d v d u + \int_{-20}^0 \int_{\Sigma_t} |\Box \phi| |\partial_t \phi| d x d t + \Vert \phi_0 \Vert_{H^1}^2 + \Vert \phi_1 \Vert_{L^2}^2 \right ]
        \end{aligned}
    \end{equation}
    \item The $\partial_t$ energy flux through the truncated outgoing cone $C_{u_0}^T$ decays in terms of appropriate norms. More precisely, we have that
    \begin{equation}
        \begin{aligned}
        \Vert \overline{\partial} \phi \Vert_{L^2 (C_{u_0}^T)}^2 \le C \int_{u_0 / 10}^{u_0} \int_u^{2 T - u} \int_{S^2} |\Box \phi| |\partial_t \phi| r^2 d \omega d v d u + {C \over \langle u_0 \rangle} \int_{u_0 / 10}^{u_0} \int_u^{2 T - u} \int_{S^2} |\Box \phi| r^2 |L(r \phi)| d \omega d v d u
        \\ + {C \over \langle u_0 \rangle^2} \left [\int_{-10}^{u_0} \int_u^{2 T - u} \int_{S^2} |\Box \phi| r^3 |L(r \phi)| d \omega d v d u + \int_{-20}^0 \int_{\Sigma_t} |\Box \phi| |\partial_t \phi| d x d t + \Vert \phi_0 \Vert_{H^1}^2 + \Vert \phi_1 \Vert_{L^2}^2 \right ]
        \end{aligned}
    \end{equation}
\end{enumerate}
\end{proposition}
\begin{proof}
The proof is essentially the same as the one found in \cite{DafRod10} (see also, for example, \cite{Yan13}). We provide a proof here for completeness.

The proof involves using the positive $r^p$ spacetime bulk integrals for $p = 2$ from Proposition~\ref{prop:r^p} in order to show that the $p = 1$ fluxes decay in $u$ on average. Repeating this argument then allows us to show that the energy flux (which corresponds to the bulk integrals for $r^p$ with $p = 1$) decay even faster on average. The average decay is then turned into the decay statement we desire by simply applying an energy estimate.

Let
\[
f_1 (u) = \int_u^{2 T - u} \int_{S^2} r (L(r \phi))^2 d \omega d v.
\]
This is simply the $r^p$ flux with $p = 1$ through the truncated outgoing cone $C_u^T$. Now, by Proposition~\ref{prop:r^p} with $p = 2$, we note that
\[
\int_{-10}^{u_0} f_1 (u) d u \le C \int_{-10}^{u_0} \int_u^{2 T - u} \int_{S^2} |\Box \phi| r^{3} |L(r \phi)| d \omega d v d u + C \int_{-10}^{2 T + 10} \int_{S^2} r^2(L(r \psi))^2 d \omega d v.
\]
We now use the fact that $r^2$ is bounded on the part of the support of $\psi$ where $u=-10$ to get
\[
C \int_{-10}^{2 T + 10} \int_{S^2} r^2(L(r \psi))^2 d \omega d v \le C \int_{-10}^{2 T + 10} \int_{S^2} (L(r \psi))^2 d \omega d v\le C \int_{-20}^0 \int_{\Sigma_t} |\Box \phi| |\partial_t \phi| d x d t + C \Vert \phi_0 \Vert_{H^1}^2 + C \Vert \phi_1 \Vert_{L^2}^2
\]
using the $\partial_t$ energy flux (the difference between $(L(r \phi))^2$ and $r^2 (L \phi)^2$ can be controlled using a similar argument to the one in Lemma~\ref{lem:r^pencomp} below).  We have that
\begin{align*}
\int_{u_0/3}^{2u_0/3} f_1 (u) d u &\le \int_{-10}^{u_0} f_1 (u)\\
&\le C \int_{-10}^{u_0} \int_u^{2 T - u} \int_{S^2} |\Box \phi| r^{3} |L(r \phi)| d \omega d v d u\\
&\qquad\qquad +C \int_{-20}^0 \int_{\Sigma_t} |\Box \phi| |\partial_t \phi| d x d t +C \Vert \phi_0 \Vert_{H^1}^2 + C \Vert \phi_1 \Vert_{L^2}^2.
\end{align*}
Now, there must exist some $\tau_1\in [u_0/3,2u_0/3]$ such that
\begin{equation} \label{eq:r^pp=1dec}
    \begin{aligned}
    f_1 (\tau_1) \le {C \over u_0} \left [\int_{-10}^{u_0} \int_u^{2 T - u} \int_{S^2} |\Box \phi| r^3 |L(r \phi)| d \omega d v d u + \int_{-20}^0 \int_{\Sigma_t} |\Box \phi| |\partial_t \phi| d x d t + \Vert \phi_0 \Vert_{H^1}^2 + \Vert \phi_1 \Vert_{L^2}^2 \right ].
    \end{aligned}
\end{equation}
This is the first step in proving decay of the $\partial_t$ energy, and it already gives us decay of the $r^p$ flux with $p = 1$. We note that $\tau_1$ is comparable to $u_0$. Thus, computing the $r^p$ with $p = 1$ flux change in the region bounded by $C_{u_0}^T$, $C_{\tau_1}^T$, and $\Sigma_T$ (that is, using Proposition~\ref{prop:r^p}) gives us that 
\begin{equation}
    \begin{aligned}
    f_1 (u_0) &\le f_1 (\tau_1) + C \int_{\tau_1}^{u_0} \int_u^{2 T - u} \int_{S^2} |\Box \phi| r^2 |L(r \phi)| d \omega d v d u
    \\ &\le C \int_{\tau_1}^{u_0} \int_u^{2 T - u} \int_{S^2} |\Box \phi| r^2 |L(r \phi)| d \omega d v d u
    \\ &\ \ + {C \over u_0} \left [\int_{-10}^{u_0} \int_u^{2 T - u} \int_{S^2} |\Box \phi| r^3 |L(r \phi)| d \omega d v d u + \int_{-20}^0 \int_{\Sigma_t} |\Box \phi| |\partial_t \phi| d x d t + \Vert \phi_0 \Vert_{H^1}^2 + \Vert \phi_1 \Vert_{L^2}^2 \right ].
    \end{aligned}
\end{equation}
This completes the proof of the decay of the $r^p$ energy flux with $p = 1$.

In order to prove that the $\partial_t$ energy flux decays, we return to \eqref{eq:r^pp=1dec}, and we now consider the interval $[\tau_1,u_0]$. The length of this interval is comparable to $u_0$. We now define
\[
f_0 (u) = \int_u^{2 T - u} \int_{S^2} (L(r \phi))^2 + |e_A (r \phi)|^2 d \omega d v.
\]
This will control the $\partial_t$ energy flux through the truncated null cone $C_u^T$. It is also the bulk integral for the $r^p$ energy with $p = 1$. Using Proposition~\ref{prop:r^p} with $p = 1$ gives us that
\[
\int_{\tau_1}^{u_0} f_0 (u) d u \le C \int_{\tau_1}^{u_0} \int_u^{2 T - u} \int_{S^2} |\Box \phi| r^2 |L(r \phi)| d \omega d v d u + C f_1 (\tau_1).
\]
Now, there must exist some $\tau_2$ in the interval $[\tau_1,u_0]$ such that
\[
f_0 (\tau_2) \le {C \over u_0} \left [\int_{\tau_1}^{u_0} \int_u^{2 T - u} \int_{S^2} |\Box \phi| r^2 |L(r \phi)| d \omega d v d u + f_1 (\tau_1) \right ].
\]

Then, by Lemma~\ref{lem:r^pencomp} below, we have that the $\partial_t$ energy flux through the truncated null cone $C_{\tau_2}^T$ is controlled by $f_0 (\tau_2)$, meaning that
\[
\Vert \overline{\partial} \phi \Vert_{L^2 (C_{\tau_2}^T)}^2 \le C f_0 (\tau_2).
\]
Then, performing a $\partial_t$ energy estimate in the region bounded by $C_{\tau_2}^T$, $C_{u_0}^T$, and $\Sigma_T$ gives us that
\begin{equation}
    \begin{aligned}
    \Vert \overline{\partial} \phi \Vert_{L^2 (C_{u_0}^T)}^2 &\le \Vert \overline{\partial} \phi \Vert_{L^2 (C_{\tau_2}^T)}^2 + C \int_{\tau_2}^{u_0} \int_u^{2 T - u} \int_{S^2} |\Box \phi| |\partial_t \phi| r^2 d \omega d v d u
    \\ &\le C f_0 (\tau_2) + C \int_{\tau_2}^{u_0} \int_u^{2 T - u} \int_{S^2} |\Box \phi| |\partial_t \phi| r^2 d \omega d v d u.
    \end{aligned}
\end{equation}
Plugging in the inequalities for $f_0 (\tau_2)$ and $f_1 (\tau_1)$ then gives us that
\begin{equation}
    \begin{aligned}
    \Vert \overline{\partial} \phi \Vert_{L^2 (C_{u_0}^T)}^2 \le C \int_{u_0 / 10}^{u_0} \int_u^{2 T - u} \int_{S^2} |\Box \phi| |\partial_t \phi| r^2 d \omega d v d u + {C \over u_0} \int_{u_0 / 10}^{u_0} \int_u^{2 T - u} \int_{S^2} |\Box \phi| r^2 |L(r \phi)| d \omega d v d u
    \\ + {C \over u_0^2} \left [\int_{-10}^{u_0} \int_u^{2 T - u} \int_{S^2} |\Box \phi| r^3 |L(r \phi)| d \omega d v d u + \int_{-20}^0 \int_{\Sigma_t} |\Box \phi| |\partial_t \phi| d x d t + \Vert \phi_0 \Vert_{H^1}^2 + \Vert \phi_1 \Vert_{L^2}^2 \right ],
    \end{aligned}
\end{equation}
which is the desired result.

\end{proof}

We now provide a technical result from \cite{DafRod10} which shows that the integrand in the $r^p$ with $p = 1$ bulk controls the $\partial_t$ energy flux through outgoing cones. We explicitly state this as a lemma for ease of reading, as we must use it in a few places.

\begin{lemma} \label{lem:r^pencomp}
Let $u_0 \ge 10$. We have that
\begin{equation} \label{eq:r^pencomp}
    \begin{aligned}
    \Vert \overline{\partial} \phi \Vert_{L^2 (C_{u_0}^T)}^2 \le C \int_{u_0}^{2 T - u_0} \int_{S^2} (L(r \phi))^2 + |\slashed{\nabla} (r \phi)|^2 d \omega d v.
    \end{aligned}
\end{equation}
\end{lemma}
\begin{proof}
The proof can be found in \cite{DafRod10}. For completeness, we provide it here.

The angular derivative terms appear in the same way in both the left hand side and right hand side of \eqref{eq:r^pencomp}, so we most only consider the difference between $(L(r \phi))^2$ present in the right hand side and $r^2 (L(\phi))^2$ present in the left hand side. We have that
\begin{equation} \label{eq:Tr^pfluxdiff}
    \begin{aligned}
    \int_{u_0}^{2 T - u_0} \int_{S^2} (L(r \phi))^2 d \omega d v = \int_{u_0}^{2 T - u_0} \int_{S^2} r^2 (L \phi)^2 + L(r \phi^2) d v
    \\ = \int_{u_0}^{2 T - u_0} \int_{S^2} r^2 (L \phi)^2 d v + \int_{S^2} r \phi^2 (2 T - u_0,u_0,\omega) d \omega - \int_{S^2} r \phi^2 (u_0,u_0,\omega) d \omega.
    \end{aligned}
\end{equation}
The third of these terms vanish by the fact that $\phi^2$ is bounded and the fact that $r(u_0,u_0,\omega) = 0$, and the second term has a good sign. From this, we get that \eqref{eq:r^pencomp} holds, as desired.
\end{proof}

This instantly allows us to prove an energy bound on other appropriate hypersurfaces which lie to the future of some $C_{u_0}^T$. The useful case of the following result will be when $\Sigma$ is a downward opening cone corresponding to the light cone of an auxiliary multiplier. This result will give us decay of the energy flux through these cones.

\begin{proposition} \label{prop:enestabovecone}
Let $\Omega$ be a region such that $\partial \Omega \subset \Sigma \cup C_{u_0}^T$ where $\Sigma$ is some Lipschitz hypersurface. Then, for $N$ the Lorentzian normal vector field of $\Sigma$ and $T [\phi]$ the energy momentum tensor associated to $\phi$, we have that
\begin{equation}
    \begin{aligned}
    \int_\Sigma T(\partial_t,N) d vol (\Sigma) \le \int_{\Omega} |\Box \phi| |\partial_t \phi| d x d t + C \int_{u_0 / 10}^{u_0} \int_u^{2 T - u} \int_{S^2} |\Box \phi| |\partial_t \phi| r^2 d \omega d v d u
    \\ + {C \over \langle u_0 \rangle} \int_{u_0 / 10}^{u_0} \int_u^{2 T - u} \int_{S^2} |\Box \phi| r^2 |L(r \phi)| d \omega d v d u
    \\ + {C \over \langle u_0 \rangle^2} \left [\int_{-10}^{u_0} \int_u^{2 T - u} \int_{S^2} |\Box \phi| r^3 |L(r \phi)| d \omega d v d u + \int_{-20}^0 \int_{\Sigma_t} |\Box \phi| |\partial_t \phi| d x d t + \Vert \phi_0 \Vert_{H^1}^2 + \Vert \phi_1 \Vert_{L^2}^2 \right ]
    \end{aligned}
\end{equation}

\end{proposition}
\begin{proof}
This follows immediately from applying the divergence theorem to the current coming from contracting $\partial_t$ with the energy momentum tensor in order to get energy estimates in the region $\Omega$. The energy on the cone $C_{u_0}^T$ is then controlled using Proposition~\ref{prop:uendec}. We note that we are allowing $\Sigma$ to be null, meaning that there may be a choice involved in finding $N$ (see \cite{DafRod08}). In practice, we shall take $N$ to be the usual null generator $L' = -\partial_t + \partial_{r'}$ for the auxiliary cones.
\end{proof}

Finally, we now use Proposition~\ref{prop:uendec} in order to prove that an averaged energy decays in $u$.

\begin{proposition} \label{prop:avenestabovecone}
Let $10 \le u_0 \le u_1$, and let $T \ge 0$. Moreover, let $\chi_u$ denote the cutoff function $\chi_{u_0 \le u \le u_1}$. Then, we have that
\begin{equation}
    \begin{aligned}
    \int_0^T \int_{\Sigma_t} \chi_u (1 + t)^{-1 - \delta} (\partial \phi)^2 d x d t \le C \int_0^T \int_{\Sigma_t} \chi_u |\Box \phi| |\partial_t \phi| d x d t
    \\ + C \int_{u_0 / 10}^{u_0} \int_u^{2 T - u} \int_{S^2} |\Box \phi| |\partial_t \phi| r^2 d \omega d v d u + {C \over \langle u_0 \rangle} \int_{u_0 / 10}^{u_0} \int_u^{2 T - u} \int_{S^2} |\Box \phi| r^2 |L(r \phi)| d \omega d v d u
    \\ + {C \over \langle u_0 \rangle^2} \left [\int_{-10}^{u_0} \int_u^{2 T - u} \int_{S^2} |\Box \phi| r^3 |L(r \phi)| d \omega d v d u + \int_{-20}^0 \int_{\Sigma_t} |\Box \phi| |\partial_t \phi| d x d t + \Vert \phi_0 \Vert_{H^1}^2 + \Vert \phi_1 \Vert_{L^2}^2 \right ]
    \end{aligned}
\end{equation}
\end{proposition}
\begin{proof}
This follows from using $(1 + t)^{-\delta} \partial_t$ as a multiplier field in between the two truncated cones $C_{u_0}^T$ and $C_{u_1}^T$. This produces a spacetime term with a good sign, which is the term that appears on the left hand side in the above. The inequality then follows from dropping other positive terms and using the fact that the energy flux on $C_{u_0}^T$ (and, thus, the flux associated to $(1 + t)^{-\delta} \partial_t$) is controlled by the right hand side using Proposition~\ref{prop:uendec}.
\end{proof}

\section{Proof of Theorem~\ref{thm:main}} \label{sec:bootstrap}
We finally have all the tools in place to prove Theorem~\ref{thm:main}. The proof proceeds in the standard way using a bootstrap argument. In Section~\ref{sec:BootstrapAssumptions}, we set up the bootstrap argument and state the bootstrap assumptions, and the rest of this section is dedicated to recovering the bootstrap assumptions. This will then complete the proof of Theorem~\ref{thm:main}.

\subsection{Bootstrap assumptions} \label{sec:BootstrapAssumptions}
We recall that $\delta > 0$ is a parameter to be chosen such that the following argument carries through. We now restrict $N_0$ to be sufficiently large as a function of $\delta$ to guarantee interpolation losses of at most ${\delta^2 \over 100}$ when interpolating between $H^{3 N_0 / 4}$ and $H^{N_0}$ to bound $H^{(3 N_0 / 4) + 20}$. In other words, $N_0$ is chosen so large such that
\[
\Vert f \Vert_{H^{(3 N_0 / 4) + 20}} \le C \Vert f \Vert_{H^{3 N_0 / 4}}^{1 - {\delta^2 \over 100}} \Vert f \Vert_{H^{N_0}}^{{\delta^2 \over 100}}
\]
for any smooth compactly supported function $f$. Thus, by Sobolev embedding on unit-size balls, and after commuting at most ${3 N_0 \over 4} + 10$ times, we may put both factors in the nonlinearity on the right hand side in $L^\infty$ and use the pointwise bootstrap assumptions with losses of size ${\delta^2 \over 100}$. We shall do this freely in the following.

We now let $T$ be the largest time such that the following bootstrap assumptions hold true.

\begin{enumerate}
\item \label{ba:badder} For $t \le T$ and $|\alpha| \le {3 N_0 \over 4}$,
\[
|\partial^\alpha \phi| + |\partial \partial^\alpha \phi| \le \epsilon^{{3 \over 4}} \frac{(1+|u|)^\delta}{(1+t)(1+|u|)}
\]
and for $|\beta| \le {3 N_0 \over 4}+15$,
\[
|\partial^\alpha \phi| + |\partial \partial^\alpha \phi| \le \epsilon^{{3 \over 4}} \left(\frac{(1+|u|)^\delta}{(1+t)(1+|u|)}\right)^{1-{\delta \over 100}}
\]
\item \label{ba:goodder} For $|\alpha| \le {3 N_0 \over 4}$ and $t \le T$,
\[
|\bar\partial \partial^\alpha \phi| \le \epsilon^{{3 \over 4}} \frac{(1+|u|)^\delta}{(1+t)^2}
\]
and for $|\beta| \le {3 N_0 \over 4}+15$,
\[
|\bar\partial\partial^\alpha \phi| \le \epsilon^{{3 \over 4}} \left(\frac{(1+|u|)^\delta}{(1+t)^2}\right)^{1-{\delta \over 100}}
\]

\item \label{ba:endec} For $|\alpha| \le N_0$ and $u_0 \ge 10$, the energy flux decays as we change appropriately truncated cones, that is,
\[
\Vert \overline{\partial} \partial^\alpha \phi \Vert_{L^2 (C_{u_0}^T)}^2 \le \epsilon^{{3 \over 2}} \langle u_0 \rangle^{-2}.
\]
Similarly, the energy flux decays as we restrict to $\Sigma_t$ hypersurfaces which are cut off farther and farther from the cone $t = r$, that is,
\[
\sup_{0 \le t \le T} \Vert \chi_{u \ge u_0} \partial \partial^\alpha \phi \Vert_{L^2 (\Sigma_t)}^2 \le \epsilon^{{3 \over 2}} \langle u_0 \rangle^{-2}.
\]

\item \label{ba:avendec} For $|\alpha| \le N_0$ and $u_0 \ge 10$, the averaged energy estimate restricted to the future of a truncated cone decays as a function of the position of the tip of the cone, that is,
\[
\int_0^T \int_{\Sigma_t} \chi_{u \ge u_0} (1 + t)^{-1 - \delta} (\partial \partial^\alpha \phi)^2 d x d t \le \epsilon^{{3 \over 2}} \langle u_0 \rangle^{-2}.
\]

\item \label{ba:r^pb} For $|\alpha| \le N_0$ and $0 \le p \le 2$, the $r^p$ energy flux is bounded on appropriately truncated cones and the $r^p$ bulk is bounded, that is,
\[
\sup_{u \ge -10} \Vert r^{{p \over 2} - 1} \partial_v (r \partial^\alpha \phi) \Vert_{L^2 (\mathcal{C}_u^T)} \le \epsilon^{{3 \over 4}}
\]
and
\[
\int_0^T \int_{\Sigma_t} \left [r^{p - 3} (L(r \phi))^2 + (2 - p) r^{p-3} |\slashed{\nabla} (r \phi)|^2 \right ] d x d t \le \epsilon^{{3 \over 2}}.
\]

\item \label{ba:r^pp=1} For $|\alpha| \le N_0$ and $u_0 \ge 10$, the $p = 1$ energy flux and bulk decay as we change appropriately truncated cones, that is,
\[
\Vert r^{-{1 \over 2}} L (r \partial^\alpha \phi) \Vert_{L^2 (\mathcal{C}_{u_0}^T)}^2 \le \epsilon^{{3 \over 2}} \langle u_0 \rangle^{-1}
\]
and
\[
\int_0^T \int_{\Sigma_t} \chi_{u \ge u_0} \left [ r^{-2} (L(r \partial^\alpha \phi))^2 + r^{-2} |\slashed{\nabla} (r \partial^\alpha \phi)|^2 \right ] d x d t \le \epsilon^{{3 \over 2}} \langle u_0 \rangle^{-1}.
\]
\item \label{ba:totalen} 
For $|\alpha|\le N_0$, 
\[
\sup_{0 \le t \le T} \Vert \partial \partial^\alpha \phi \Vert_{L^2 (\Sigma_t)}^2 \le \epsilon^{{3 \over 2}}.
\]

\end{enumerate}

We now turn to recovering the bootstrap assumptions. Because $\epsilon$ is allowed to depend on $\delta$, we note that we can assume without loss of generality that $T \ge \delta^{-10}$. Moreover, we shall freely use the results of Lemma~\ref{lem:nullforms}, which says that our class of null forms satisfies the usual improved estimates even after commuting with translation vector fields.

We shall refer to the first two bootstrap assumptions as the pointwise bootstrap assumptions, and we shall refer to the remaining bootstrap assumptions as the energy bootstrap assumptions. We shall first recover the energy bootstrap assumptions in Sections~\ref{sec:r^pbootstrrec} through \ref{sec:enbootstrap} before turning to the pointwise bootstrap assumptions in Sections~\ref{sec:goodderivrecovery} through \ref{sec:badderivrecovery}. Section~\ref{sec:endeclem} contains an estimate on the decay of energy for the solution on null cones adapted to auxiliary multipliers. It is used when recovering the pointwise bootstrap assumptions, and it is singled out for its importance.

\subsection{Recovery of $r^p$ bootstrap assumptions} \label{sec:r^pbootstrrec}
In this section, we shall recover the bootstrap assumptions involving the $r^p$ energies. This involves showing that the $r^p$ energies and bulks are bounded for $0 \le p \le 2$, as well as showing that the $r^p$ with $p = 1$ flux and bulk decay appropriately in $u$. The boundedness of the $r^p$ energies for $0 \le p \le 2$ is the subject of Section~\ref{sec:r^pbound}, while the decay of the $r^p$ with $p = 1$ flux and bulk is the subject of Section~\ref{sec:r^pdec}.

\subsubsection{Boundedness of $r^p$ energies for $0 \le p \le 2$} \label{sec:r^pbound}

We now recover the bootstrap assumption involving boundedness of the $r^p$ energies with $0 \le p \le 2$. Because we are only interested in boundedness, we shall only consider the case $p = 2$, as other values of $p$ are easier to control.

The nonlinear error integrals are of the form
\begin{align}
    \sum_{\mu + \beta + \delta = \alpha} \int_0^s \int_{\Sigma_s} (\partial^\mu f) m(d \partial^\beta \phi,d \partial^\gamma \phi) r L(r \partial^\alpha \phi) d x d t = \int_0^s \int_0^\infty \int_{S^2} \partial^\alpha (f m(d \phi,d \phi)) r^3 L(r \partial^\alpha \phi) d \omega d r d t.
\end{align}
Now, as in Lemma~\ref{lem:nullforms}, we note that
\begin{equation}
    \begin{aligned}
    |(\partial^\mu f) r m(d \partial^\beta \phi,d \partial^\gamma \phi)| &\le |\partial \partial^\beta \phi| |r \bar \partial \partial^\gamma \phi| + |r \bar \partial \partial^\beta \phi| |\partial \partial^\gamma \phi|
    \\ &\le |\partial \partial^\beta \phi| (|\bar \partial (r \partial^\gamma \phi)| + |\partial^\gamma \phi|) + |\partial \partial^\gamma \phi| (|\overline{\partial} (r \partial^\beta \phi)| + |\partial^\beta \phi|).
    \end{aligned}
\end{equation}
Thus, it suffices to control integrals of the form
\begin{equation} \label{eq:r^perr1}
    \begin{aligned}
    \int_0^s \int_0^\infty \int_{S^2} |\partial \partial^\beta \phi| |\bar \partial (r \partial^\gamma \phi)| r^2 |L(r \partial^\alpha \phi)| d \omega d r d t
    \end{aligned}
\end{equation}
along with integrals of the form
\begin{equation} \label{eq:r^perr2}
    \begin{aligned}
    \int_0^s \int_0^\infty \int_{S^2} |\partial \partial^\beta \phi| |\partial^\gamma \phi| r^2 |L(r \partial^\alpha \phi)| d \omega d r d t.
    \end{aligned}
\end{equation}
We consider two cases depending on whether $|\gamma| \le |\beta|$ or $|\beta| < |\gamma|$.

We first assume that $|\gamma| \le |\beta|$. In this case, we use pointwise estimates for the term with $\partial^\gamma$ and $L^2$ estimates for the term with $\partial^\beta$. To effectively use the fact that the energy decays in $u$, we decompose dyadically by setting $\tau_k = 2^k$ and $\tau_{-1} = -10$, and we set $N(s) = \lceil \log(s) \rceil - 1$. We moreover denote by $\chi_k$ the characteristic function $\chi_{\tau_k \le u \le \tau_{k + 1}}$ restricting to this dyadic region in $u$. Considering first integrals of the form \eqref{eq:r^perr1}, we have that
\begin{align}
    \nonumber &\int_0^s \int_0^\infty \int_{S^2} |\partial \partial^\beta \phi| |\bar \partial (r \partial^\gamma \phi)| r^2 |L(r \partial^\alpha \phi)| d \omega d r d t
    \\ \nonumber &= \sum_{k = -1}^{N(s)} \int_0^s \int_0^\infty \int_{S^2} \chi_{\tau_k \le u \le \tau_{k + 1}} |\partial \partial^\beta \phi| |\bar \partial (r \partial^\gamma \phi)| r^2 |L(r \partial^\alpha \phi)| d \omega d r d t
    \\ \label{eq:Term1r^p} &\le \sum_{k = -1}^{N(s)} \langle \tau_k \rangle^{-1} \Vert \chi_k \langle t \rangle^{-{1 \over 2} - \delta} r \langle \tau_k \rangle \partial \partial^\beta \phi \Vert_{L_u^2 L_v^2 L_\omega^2} \Vert \chi_k \langle t \rangle^{{1 \over 2} + \delta} \bar \partial (r \partial^\gamma \phi) \Vert_{L_u^2 L_v^\infty L_\omega^\infty} \Vert r L(r \partial^\alpha \phi) \Vert_{L_u^\infty L_v^2 L_\omega^2}.
\end{align}
The first and third terms in this product are each bounded by $C \epsilon^{{3 \over 4}}$ using bootstrap assumptions \ref{ba:avendec} and \ref{ba:r^pb}, respectively. To evaluate the second term, we plug in the pointwise bootstrap assumptions \ref{ba:badder} and \ref{ba:goodder} and integrate to find that
\begin{equation} \label{eq:r^perrterm1}
    \begin{aligned}
    \Vert \chi_k \langle t \rangle^{{1 \over 2} + \delta} \bar \partial (r \partial^\gamma \phi) \Vert_{L_u^2 L_v^\infty L_\omega^\infty} \le C \langle \tau_k \rangle^{10 \delta} \epsilon^{{3 \over 4}}.
    \end{aligned}
\end{equation}
This allows us to evaluate the sum, giving us that \eqref{eq:Term1r^p} is bounded by
\[
C \sum_{k = -1}^{N(s)} \langle \tau_k \rangle^{-1 + 10 \delta} \epsilon^{{9 \over 4}} \le C \epsilon^{{9 \over 4}},
\]
as desired. An analogous argument works for integrals of the form \eqref{eq:r^perr2}, with the only difference being that we use the estimate
\[
\Vert \chi_k \langle t \rangle^{{1 \over 2} + \delta} \partial^\gamma \phi \Vert_{L_u^2 L_v^\infty L_\omega^\infty} \le C \langle \tau_k \rangle^{10 \delta} \epsilon^{{3 \over 4}}
\]
instead of \eqref{eq:r^perrterm1}.

We now assume that $|\beta| < |\gamma|$. In this case, we use pointwise estimates for the term with $\partial^\beta$ and $L^2$ estimates for the term with $\partial^\gamma$. We have that
\begin{align*}
    \int_0^s \int_0^\infty \int_{S^2} |\partial \partial^\beta \phi| |\bar \partial (r \partial^\gamma \phi)| r^2 |L(r \partial^\alpha \phi)| d \omega d r d t
    \\ \le \Vert r^{{1 \over 2} + \delta} \partial \partial^\beta \phi \Vert_{L_u^2 L_v^\infty L_\omega^\infty} \Vert r^{{1 \over 2} - \delta} \overline{\partial} (r \partial^\gamma \phi) \Vert_{L_u^2 L_v^2 L_\omega^2} \Vert r L(r \partial^\alpha \phi) \Vert_{L_u^\infty L_v^2 L_\omega^2},
\end{align*}
and we can directly see that all of these terms are appropriately bounded by the bootstrap assumptions. Similarly, we have that
\begin{align}
    \int_0^s \int_0^\infty \int_{S^2} |\partial \partial^\beta \phi| |\partial^\gamma \phi| r^2 |L(r \partial^\alpha \phi)| d \omega d r d t
    \\ \le \Vert r^{{1 \over 2} + \delta} \partial \partial^\beta \phi \Vert_{L_u^2 L_v^\infty L_\omega^\infty} \Vert r^{{1 \over 2} - \delta} \partial^\gamma \phi \Vert_{L_u^2 L_v^2 L_\omega^2} \Vert r L(r \partial^\alpha \phi) \Vert_{L_u^\infty L_v^2 L_\omega^2},
\end{align}
and once again, all of these terms are bounded by the bootstrap assumptions.

\subsubsection{Decay in $u$ of $r^p$ energy with $p = 1$} \label{sec:r^pdec}

We finally control the $r^p$ energy with $p = 1$ and recover bootstrap assumption \ref{ba:r^pp=1}. This involves essentially showing that the $r^p$ energy and bulk coming from $p = 1$ decay like $u^{-1}$. An application of Proposition~\ref{prop:uendec} tells us that it suffices to show that
\begin{equation} \label{eq:r^pp=1eq1}
    \begin{aligned}
    \int_{u_0 / 10}^{u_0} \int_u^{2 T - u} \int_{S^2} |\Box \partial^\alpha \phi| r^2 |L(r \partial^\alpha \phi)| d \omega d v d u \le {C \over \langle u_0 \rangle} \epsilon^2,
    \end{aligned}
\end{equation}
and that
\begin{equation} \label{eq:r^pp=1eq2}
    \begin{aligned}
    \left [\int_{-10}^{u_0} \int_u^{2 T - u} \int_{S^2} |\Box \partial^\alpha \phi| r^3 |L(r \partial^\alpha \phi)| d \omega d v d u + \int_{-20}^0 \int_{\Sigma_t} |\Box \partial^\alpha \phi| |\partial_t \phi| d x d t + \Vert \phi_0 \Vert_{H^{|\alpha| + 1}}^2 + \Vert \phi_1 \Vert_{H^{|\alpha|}}^2 \right ] \le C \epsilon^2.
    \end{aligned}
\end{equation}
The third and fourth terms in \eqref{eq:r^pp=1eq2} are bounded appropriately by data, while the second term can be controlled using the powers of $\epsilon$ in the bootstrap assumptions (i.e., without using decay) because the integrand is supported in a set of finite spacetime volume. Finally, the first term in \eqref{eq:r^pp=1eq2} was controlled when we were recovering the boundedness of the $r^p$ with $p = 2$ energy (see Section~\ref{sec:r^pbound}). Thus, it suffice to prove the bound in \eqref{eq:r^pp=1eq1}.

Proceeding as in Section~\ref{sec:r^pbound}, we note that it suffices to consider integrals of the form
\begin{equation} \label{eq:r^perr11}
    \begin{aligned}
    \int_{u_0 / 10}^{u_0} \int_u^{2 T - u} \int_{S^2} |\partial \partial^\beta \phi| |\bar \partial (r \partial^\gamma \phi)| r |L(r \partial^\alpha \phi)| d \omega d v d u,
    \end{aligned}
\end{equation}
along with integrals of the form
\begin{equation} \label{eq:r^perr12}
    \begin{aligned}
    \int_{u_0 / 10}^{u_0} \int_u^{2 T - u} \int_{S^2} |\partial \partial^\beta \phi| |\partial^\gamma \phi| r |L(r \partial^\alpha \phi)| d \omega d v d u.
    \end{aligned}
\end{equation}

We first suppose that $|\gamma| \le |\beta|$. We then use the pointwise bootstrap assumptions on the term with $|\gamma|$ derivatives. Considering first integrals of the form \eqref{eq:r^perr11}, we have that
\begin{equation}
    \begin{aligned}
    \int_{u_0 / 10}^{u_0} \int_u^{2 T - u} \int_{S^2} |\partial \partial^\beta \phi| |\overline{\partial} (r \partial^\gamma \phi)| r |L(r\partial^\alpha \phi)| d \omega d v d u
    \\ \le \langle u_0 \rangle^{-1} \Vert \chi_{u \ge u_0 / 10} \langle u_0 \rangle r \partial \partial^\beta \phi \Vert_{L_t^\infty L_r^2 L_\omega^2} \Vert \overline{\partial} (r \partial^\gamma \phi) \Vert_{L_t^2 L_r^\infty L_\omega^\infty} \Vert L(r \partial^\alpha \phi) \Vert_{L_t^2 L_r^2 L_\omega^2}.
    \end{aligned}
\end{equation}
The first and third term are bounded by $C \epsilon^{{3 \over 4}}$ using bootstrap assumptions \ref{ba:endec} and \ref{ba:r^pb}, respectively. The pointwise bootstrap assumptions \ref{ba:badder} and \ref{ba:goodder} and integrating directly control the second term by $C \epsilon^{{3 \over 4}}$. Thus, the whole term is controlled by
\[
C \epsilon^{{9 \over 4}} \langle u_0 \rangle^{-1},
\]
as desired. We note that we are in fact not using additional powers of $\langle u_0 \rangle$ that we otherwise could, and we shall continue to do so throughout this section for simplicity. An analogous argument works for integrals of the form \eqref{eq:r^perr12}.

We now consider the case where $|\beta| \le |\gamma|$. For integrals of the form of \eqref{eq:r^perr11}, we have that
\begin{equation}
    \begin{aligned}
    \int_{u_0 / 10}^{u_0} \int_u^{2 T - u} \int_{S^2} |\partial \partial^\beta \phi| |\overline{\partial} (r \partial^\gamma \phi)| r |L(r \phi)| d \omega d v d u
    \\ \le \Vert \chi_{u \ge u_0 / 10} r^\delta \partial \partial^\beta \phi \Vert_{L_u^\infty L_v^\infty L_\omega^\infty} \Vert r^{{1 \over 2} - \delta} \overline{\partial} (r \partial^\gamma \phi) \Vert_{L_u^2 L_v^2 L_\omega^2} \Vert r^{{1 \over 2}} L(r \partial^\alpha \phi) \Vert_{L_u^2 L_v^2 L_\omega^2}.
    \end{aligned}
\end{equation}
Now, we note that
\[
\Vert \chi_{u \ge u_0 / 10} r^\delta \partial \partial^\beta \phi \Vert_{L_u^\infty L_v^\infty L_\omega^\infty} \le C \epsilon^{{3 \over 4}} \langle u_0 \rangle^{-2 + 2 \delta}
\]
using the pointwise bootstrap assumption \ref{ba:badder} because $t \ge u - 1$ in the support of $\phi$, and because the cutoff function restricts us to the region where $u \ge u_0 / 10$. Thus, using bootstrap assumption \ref{ba:r^pb} on the $r^p$ energy boundedness for the second and third terms, the error integral in question is controlled
\[
C \epsilon^{{9 \over 4}} \langle u_0 \rangle^{-1},
\]
as desired. For integrals of the form of \eqref{eq:r^perr12}, we have that
\begin{equation}
    \begin{aligned}
    \int_{u_0 / 10}^{u_0} \int_u^{2 T - u} \int_{S^2} |\partial \partial^\beta \phi| |\partial^\gamma \phi| r |L(r \phi)| d \omega d v d u
    \\ \le \langle u_0 \rangle^{-1} \Vert \chi_{u \ge u_0 / 10} \partial \partial^\beta \phi \Vert_{L_t^2 L_r^\infty L_\omega^\infty} \Vert \chi_{u \ge u_0 / 10} \langle u_0 \rangle r \partial^\gamma \phi \Vert_{L_t^\infty L_r^2 L_\omega^2} \Vert_{L_u^2 L_v^2 L_\omega^2} \Vert L(r \partial^\alpha \phi) \Vert_{L_t^2 L_r^2 L_\omega^2}.
    \end{aligned}
\end{equation}
Thus, using bootstrap assumptions \ref{ba:badder}, \ref{ba:endec}, and \ref{ba:r^pb}, this is controlled by
\[
C \epsilon^{{9 \over 4}} \langle u_0 \rangle^{-1},
\]
as desired.

\subsection{Recovery of remaining energy bootstrap assumptions} \label{sec:enbootstrap}

We now recover the bootstrap assumptions involving the decay in $u$ of the $\partial_t$ energy fluxes (bootstrap assumption \ref{ba:endec}), the averaged energy estimate (bootstrap assumption \ref{ba:avendec}), and the boundedness of the energy (bootstrap assumption \ref{ba:totalen}). The first step is recovering the first part of bootstrap assumption~\ref{ba:endec}, which will involve showing that the $\partial_t$ energy flux of the solution through a truncated outgoing cone $C_{u_0}^T$ for some $u_0 \ge 10$ is of size $C \epsilon^2 \langle u_0 \rangle^{-2}$. We of course also have to show this after commutation.

We commute the equation for $\phi$ as in Section~\ref{sec:r^pbound}. An application of the second part of Proposition~\ref{prop:uendec} tells us that the $\partial_t$ energy flux of $\partial^\alpha \phi$ through a truncated outgoing null cone $C_{u_0}^T$ satisfies a favorable bound as long as we control the error terms appropriately. The terms involving data are of the appropriate size by assumption. Thus, it suffices to show that
\begin{equation} \label{eq:uendecerr1}
    \begin{aligned}
    \int_{u_0 / 10}^{u_0} \int_u^{2 T - u} \int_{S^2} |\Box \partial^\alpha \phi| |\partial_t \partial^\alpha \phi| r^2 d \omega d v d u \le C \epsilon^2 \langle u_0 \rangle^{-2},
    \end{aligned}
\end{equation}
and
\begin{equation} \label{eq:uendecerr2}
    \begin{aligned}
    \int_{u_0 / 10}^{u_0} \int_u^{2 T - u} \int_{S^2} |\Box \partial^\alpha \phi| r^2 |L(r \partial^\alpha\phi)| d \omega d v d u \le C \epsilon^2 \langle u_0 \rangle^{-1},
    \end{aligned}
\end{equation}
and
\begin{equation} \label{eq:uendecerr3}
    \begin{aligned}
    \int_{-20}^0 \int_{\Sigma_t} |\Box \partial^\alpha\phi| |\partial_t \partial^\alpha\phi| d x d t \le C \epsilon^2,
    \end{aligned}
\end{equation}
and finally that
\begin{equation} \label{eq:uendecerr4}
    \begin{aligned}
    \int_{-10}^{u_0} \int_u^{2 T - u} \int_{S^2} |\Box \partial^\alpha \phi| r^3 |L(r \partial^\alpha\phi)| d \omega d v d u \le C \epsilon^2.
    \end{aligned}
\end{equation}
The term \eqref{eq:uendecerr3} can easily be bounded because it is supported in a set of finite spacetime volume. The term \eqref{eq:uendecerr4} was already appropriately controlled when recovering the boundedness of the $r^p$ energies with $p = 2$ in Section~\ref{sec:r^pbound}, and the term \eqref{eq:uendecerr2} was controlled when proving that the $p = 1$ flux decays like $u^{-1}$ in Section~\ref{sec:r^pdec}. Thus, it suffices to control the first term \eqref{eq:uendecerr1}.

Expanding out \eqref{eq:uendecerr1}, it suffices to control integrals of the form
\begin{equation} \label{eq:endecintfinal}
    \begin{aligned}
    \int_0^T \int_{\Sigma_t} \chi_{u \ge u_0 / 10} |\partial \partial^\beta \phi| |\overline{\partial} \partial^\gamma \phi| |\partial_t \partial^\alpha \phi| d x d t.
    \end{aligned}
\end{equation}
We first consider the case where $|\beta| \le |\gamma|$. We have that
\begin{equation}
    \begin{aligned}
    \int_0^T \int_{\Sigma_t} \chi_{u \ge u_0 / 10} |\partial \partial^\beta \phi| |\overline{\partial} \partial^\gamma \phi| |\partial_t \partial^\alpha \phi| d x d t
    \\ \le \Vert \chi_{u \ge u_0 / 10} \partial \partial^\beta \phi \Vert_{L_t^2 L_x^\infty} \Vert \chi_{u \ge u_0 / 10} \overline{\partial} \partial^\gamma \phi \Vert_{L_t^2 L_x^2} \Vert \chi_{u \ge u_0 / 10} \partial_t \partial^\alpha \phi \Vert_{L_t^\infty L_x^2}.
    \end{aligned}
\end{equation}
By Lemma~\ref{lem:r^pencomp}, we have that $\Vert \chi_{u \ge u_0 / 10} \overline{\partial} \partial^\gamma \phi \Vert_{L_t^2 L_x^2}$ is controlled by the square root of the $r^p$ with $p = 1$ bulk in the region where $u \ge u_0 / 10$, meaning that it is controlled by $\epsilon^{{3 \over 4}} \langle u_0 \rangle^{-{1 \over 2}}$ by bootstrap assumption \ref{ba:r^pp=1}. Then, we have that
\[
\Vert \chi_{u \ge u_0 / 10} \partial_t \partial^\alpha \phi \Vert_{L_t^\infty L_x^2} \le \epsilon^{{3 \over 4}} \langle u_0 \rangle^{-1}
\]
by bootstrap assumption \ref{ba:endec}. Finally, using the pointwise bootstrap assumption \ref{ba:badder}, we have that
\[
\Vert \chi_{u \ge u_0 / 10} \partial \partial^\beta \phi \Vert_{L_t^2 L_x^\infty} \le \epsilon^{{3 \over 4}} \langle u_0 \rangle^{-1 + \delta} \left (\int_{u_0 / 10}^T {1 \over (1 + t)^2} d t \right )^{{1 \over 2}} \le C \epsilon^{{3 \over 4}} \langle u_0 \rangle^{-{3 \over 2} + \delta}.
\]
Thus, putting everything together, we have that the integral in \eqref{eq:endecintfinal} is controlled by $C \epsilon^{{9 \over 4}} \langle u_0 \rangle^{-2}$ when $|\beta| \le |\gamma|$ (we are even dropping powers of $\langle u_0 \rangle$ in this), as desired.

We finally consider the case where $|\gamma| \le |\beta|$. We have that
\begin{equation}
    \begin{aligned}
    \int_0^T \int_{\Sigma_t} \chi_{u \ge u_0 / 10} |\partial \partial^\beta \phi| |\overline{\partial} \partial^\gamma \phi| |\partial_t \partial^\alpha \phi| d x d t
    \\ \le \Vert \chi_{u \ge u_0 / 10} \partial \partial^\beta \phi \Vert_{L_t^\infty L_x^2} \Vert \chi_{u \ge u_0 / 10} \overline{\partial} \partial^\gamma \phi \Vert_{L_t^1 L_x^\infty} \Vert \chi_{u \ge u_0 / 10} \partial_t \partial^\alpha \phi \Vert_{L_t^\infty L_x^2}.
    \end{aligned}
\end{equation}
By the bootstrap assumption \ref{ba:endec}, both of the energy quantities are controlled by $\epsilon^{{3 \over 4}} \langle u_0 \rangle^{-1}$. Moreover, bootstrap assumption \ref{ba:goodder} giving improved decay of good derivatives implies that the remaining term is bounded by $C \epsilon^{{3 \over 4}}$. Thus, this term is once again bounded by $C \epsilon^{{9 \over 4}} \langle u_0 \rangle^{-2}$, as desired.

This allows us to recover bootstrap assumptions \ref{ba:endec}, \ref{ba:avendec}, and \ref{ba:totalen}. bootstrap assumptions \ref{ba:endec} and \ref{ba:avendec} follow immediately from applying Proposition~\ref{prop:enestabovecone} and Proposition~\ref{prop:avenestabovecone}, respectively, to appropriate regions (i.e., to a region bounded from below by the truncated outgoing null cone $C_{u_0}^T$), and bootstrap assumption \ref{ba:totalen} is similar to bootstrap assumption \ref{ba:endec}. Indeed, after applying these results, we must simply control the error integrals, which can be controlled in the same way as \eqref{eq:endecintfinal} above. We note that the error integrals for bootstrap assumption \ref{ba:totalen} are of course the easiest among these, as we do not have to worry about powers of $\langle u_0 \rangle$.

\subsection{Decay of energy on auxiliary cones} \label{sec:endeclem}
We now provide an estimate for the size of the energy of $\partial^\alpha \phi$ on an auxiliary cone for all $|\alpha| \le N_0$. Although the proof follows in the same way as those in Section~\ref{sec:enbootstrap}, we state it here as a lemma because of its significance in recovering the pointwise bootstrap assumptions. Let $\hat{C}_{\hat{u}_0}^{s,a})$ be the downward opening truncated cone, defined analogously to the upward opening truncated cone, but with time reversed and with the vertex located at $(s, a, 0, 0)$.

\begin{lemma} \label{lem:endecauxcone}
Let $|\alpha| \le N_0$, let $\hat{u}_0$ be arbitrary, and let $u_0 \ge -20$. Assuming the bootstrap assumptions in Section~\ref{sec:BootstrapAssumptions}, we have that
\[
\Vert \chi_{u \ge u_0} \overline{\partial}' \partial^\alpha \phi \Vert_{L^2 (\hat{C}_{\hat{u}_0}^T)}^2 \le C \epsilon^{{3 \over 2}} \langle u_0 \rangle^{-2}.
\]
\end{lemma}
\begin{proof}
We perform a $\partial_t$ energy estimate in the region contained below the downward opening truncated null cone $\hat{C}_{\hat{u}_0}^{s,a}$ and above the upward opening truncated null cone $C_{u_0}^T$ This gives us that
\[
\Vert \chi_{u \ge u_0} \hat{\partial} \partial^\alpha \phi \Vert_{L^2 (\hat{C}_{\hat{u}_0}^{s,a})}^2 \le \Vert \overline{\partial} \partial^\alpha \phi \Vert_{L^2 (C_{u_0}^T)}^2 + \int_0^T \int_{\Sigma_t} \chi_{u \ge u_0} \chi_{\hat{u} \ge \hat{u}_0} |\Box \partial^\alpha \phi| |\partial_t \partial^\alpha \phi| d x d t.
\]
By the bootstrap assumptions, the first term is bounded by $\epsilon^{{3 \over 2}} \langle u_0 \rangle^{-2}$. The error integral can then be controlled in the same way as \eqref{eq:endecintfinal}, giving us the desired result.
\end{proof}

\subsection{Recovery of bootstrap assumption~\ref{ba:goodder}} \label{sec:goodderivrecovery}

This bootstrap assumption contains pointwise decay of $\overline{\partial} \partial^\alpha \phi$ for all $|\alpha| \le {3 N_0 \over 4}$.  We thus begin by commuting the equation \eqref{eq:mainsimple} with $\partial^\alpha$ for some $|\alpha| \le {3 N_0 \over 4}$, giving us the equation
\[
\Box \partial^\alpha \phi = \partial^\alpha (f g(d \phi,d \phi)).
\]
We now wish to prove pointwise decay for $\overline\partial\partial^\alpha \phi$ and improve the power of $\epsilon$ in the process. This will recover the bootstrap assumption. In order to do so, we shall apply Proposition~\ref{prop:goodderpointwise}. We thus pick some $s$ with $0 \le s \le T$ and a point within $\Sigma_s$ where we wish to recover pointwise decay. Without loss of generality, we recall that we can take this point to lie in the $(t,x)$ plane, meaning that it is of the form $(s,a,0,0)$. We also recall that $\tau$, the distance from the usual light cone $t = r$, is given by $s - a$. We finally recall that the data for $\psi$ are given by the pair $(\psi_0,\psi_1)$, and they are supported in the unit ball centered at $(s,a,0,0)$ in $\Sigma_s$.

In this setting, the estimate in Proposition~\ref{prop:goodderpointwise} takes the form
\begin{equation}
    \begin{aligned}
    |\overline{\partial} \partial^\alpha \phi| (s,\tau,0,0) \le {C \over (1 + s)^2} \chi_{|\tau| \le 5} \Vert \partial^\alpha \phi \Vert_{H^{11} (\Sigma_0)}
    \\ + C \sup_{|\beta| \le 10} \sup_{\psi_1} \int_0^s \int_{\Sigma_t} |\partial^\beta \partial^\alpha (f g(d \phi,d \phi))| (|\partial_y \psi| + |\partial_z \psi| + |(\partial_t + \partial_x) \psi|) d x d t.
    \end{aligned}
\end{equation}
Now, we have that
\[
{C \over (1 + s)^2} \chi_{|\tau| \le 5} \Vert \partial^\alpha \phi \Vert_{H^{11} (\Sigma_0)} \le {C \epsilon \over (1 + s)^2} \chi_{|\tau| \le 5},
\]
meaning that this term is appropriately controlled in order to recover the bootstrap assumption. All that remains is to control the error integral given by
\begin{equation} \label{eq:errorintgoodder}
    \begin{aligned}
    \sup_{|\beta| \le 10} \sup_{\psi_1} \int_0^s \int_{\Sigma_t} |\partial^\beta \partial^\alpha (f g(d \phi,d \phi))| (|\partial_y \psi| + |\partial_z \psi| + |(\partial_t + \partial_x) \psi|) d x d t.
    \end{aligned}
\end{equation}
By Proposition~\ref{prop:planeder}, we can write that
\begin{equation}\label {eq:goodderintegral}
    \begin{aligned}
    \sup_{|\beta| \le 10} \sup_{\psi_1} \int_0^s \int_{\Sigma_t} |\partial^\beta \partial^\alpha (f g(d \phi,d \phi))| (|\partial_y \psi| + |\partial_z \psi| + |(\partial_t + \partial_x) \psi|) d x d t
    \\ \le C  \sup_{|\gamma|\le 3N_0/4+10}\int_0^s \int_{\Sigma_t} \partial^{\gamma}g(d \phi,d \phi)\left(|\overline{\partial}' \psi| + {\sqrt{\tau - u - u'} \min(\sqrt{1 + t},\sqrt{1 + s - t}) \over (1 + s - t)}|\partial \psi|\right)dxdt.
    \end{aligned}
\end{equation}
We will first deal in Section~\ref{sec:somethingsmallgoodder} with the regions where that $s\le 200, t\le 100$, or $s-t\le 100$ and will assume that is not the case in later sections. Then we defer the case that $a\le s/10$ to section~\ref{sec:taularge}, so for Sections~\ref{sec:usmall} through \ref{subsec:penultimate}, we shall assume that $a \ge {s \over 10}$ (or, equivalently, that $\tau \le {s \over 10}$). This will allow us to say that the volume form in the $(t,u,u',\theta)$ has a power of ${1 \over s}$, as it is given by ${r r' \over 2 a} d t d u d u' d \theta$ (see Proposition~\ref{prop:rr'coordinates}). In Section~\ref{sec:usmall}, we deal with the regions where $|u|\le 20$ and will assume that is not the case in later sections.

After that, we divide the region of integration of \eqref{eq:goodderintegral} and take cases depending on the values of $s,\tau$. We use that in those regions where $t\le s/2$, we have
\[
|\partial_y \psi| + |\partial_z \psi| + |(\partial_t + \partial_x) \psi|\le \frac{\sqrt{t\langle\tau\rangle}}{s^2}
\]
while in those regions where $t\ge s/2$, we have
\[
|\partial_y \psi| + |\partial_z \psi| + |(\partial_t + \partial_x) \psi|\le \frac{\langle\tau^{1/2}\rangle}{(s-t)^{3/2}}.
\]
Note that outside of Section~\ref{sec:somethingsmallgoodder}, we will convert from $(t,x,y,z)$ coordinates to $(t,u,u',\theta)$. When $t,s-t\ge 100$ and  $\tau<s/10$, the factor $\frac{rr'}{a}$ that appears in the conversion (see Proposition~\ref{prop:rr'coordinates}) is bounded by $C\min(t,s-t)$.

In the following, we shall often drop explicitly writing the regions of integration from line to line for ease of notation, but we shall keep them implicitly.

\subsubsection{$s\le 200$, $t\le 100$, or $s-t\le 100$}\label{sec:somethingsmallgoodder}

In the case that $s\le 200$ or $s>200$ and $t\le 100$, we use that
\[
|\partial_y \psi| + |\partial_z \psi| + |(\partial_t + \partial_x) \psi|\le\frac{C}{(1+s)^2}
\]
and put both derivatives of $\phi$ in the integrand of \eqref{eq:errorintgoodder} in $L^2$  and use bootstrap assumption~\eqref{ba:totalen} to get the bound
\[
\frac{C\epsilon^{3/2}}{(1+s)^2}
\]

In the case that $s-t\le 100$, we use pointwise bounds on all three factors, bound the derivatives of $\psi$ by $C$, and note that the volume of integration is bounded by a constant to get
\begin{equation}\nonumber
C  \sup_{|\gamma|\le 3N_0/4+10}\int_{s-100}^{s} \int_{\Sigma_t} \partial^{\gamma}g(d \phi,d \phi)|\partial \psi|dxdt\le C\epsilon^{3/2} \left(\frac{1}{(1+s)}\right)^{1-\delta^2}\left(\frac{(1+s)^\delta}{(1+s)^2}\right)^{1-\delta^2}<C\epsilon\frac{\langle \tau\rangle^\delta}{(1+s)^2}
\end{equation}

Thus we only need to bound
\begin{equation}\label{firstintegraltobound2v2}
C  \sup_{|\gamma|\le 3N_0/4+10}\int_{100}^{s-100} \int_{\Sigma_t} \partial^{\gamma}g(d \phi,d \phi)\left(|\overline{\partial}' \psi| + {\sqrt{\tau - u - u'} \min(\sqrt{1 + t},\sqrt{1 + s - t}) \over (1 + s - t)}|\partial \psi|\right)dxdt.
\end{equation}

\subsubsection{$|u|\le 20$}\label{sec:usmall}
In this section , we bound the integral over those regions where $|u|\le 20$. Note that this entirely handles the case that $\tau\le 10$. We need to bound
\begin{align*}
C\epsilon^{3/2}\int_{100\vee (\tau-30)}^{s-100}\iint_{\{|u'|\le 1, |u|\le 20\}} \partial^{\gamma}g(d \phi,d \phi)\left(|\overline{\partial}' \psi| + {\sqrt{\tau - u - u'} \min(\sqrt{1 + t},\sqrt{1 + s - t}) \over (1 + s - t)}|\partial \psi|\right)      \min(t,s-t) d\theta du du'dt.
\end{align*}
We split into regions where $t<s/2$ and where $t>s/2$. When $t<s/2$, we use pointwise bounds on derivatives of $\phi$ to get
\begin{align*}
C\epsilon^{3/2}&\int_{100\vee (\tau-30)}^{s/2}\iint_{\{|u'|\le 1, |u|\le 20\}} \partial^{\gamma}g(d \phi,d \phi)\left(|\overline{\partial}' \psi| + {\sqrt{\tau - u - u'} \min(\sqrt{1 + t},\sqrt{1 + s - t}) \over (1 + s - t)}|\partial \psi|\right)      t d\theta du du'dt\\
&\le
C\epsilon^{3/2}\int_{100\vee (\tau-30)}^{s/2}\iint_{\{|u'|\le 1, |u|\le 20\}}    \left(\frac{\sqrt{t\langle\tau\rangle}}{s^2} \right) \left(\frac{\langle u\rangle^\delta}{t\langle u\rangle}\right)^{1-\delta^2/100}  \left( \frac{\langle u\rangle^\delta}{t^{2}}\right)^{1-\delta^2/100}       t d\theta du du'dt\\
&\le
\frac{C\epsilon^{3/2}\langle \tau\rangle^{1/2}}{s^2}\int_{100\vee (\tau-30)}^{s/2}\iint_{\{|u'|\le 1, |u|\le 20\}}t^{3\delta^2/100-3/2} d\theta du du'dt\\
&\le \frac{C\epsilon^{3/2}\langle\tau\rangle^{3\delta^2/100}}{s^2}<\frac{C\epsilon^{3/2}\langle\tau\rangle^\delta}{s^2}.
\end{align*}
When $t\ge s/2$, we use pointwise bounds on derivatives of $\phi$ to get
\begin{align*}
C\epsilon^{3/2}&\int_{s/2}^{s-100}\iint_{\{|u'|\le 1, |u|\le 20\}} \partial^{\gamma}g(d \phi,d \phi)\left(|\overline{\partial}' \psi| + {\sqrt{\tau - u - u'} \min(\sqrt{1 + t},\sqrt{1 + s - t}) \over (1 + s - t)}|\partial \psi|\right)      (s-t) d\theta du du'dt\\
&\le
C\epsilon^{3/2}\int_{s/2}^{s-100}\iint_{\{|u'|\le 1, |u|\le 20\}}    \left(\frac{\langle\tau\rangle^{1/2}}{(s-t)^{3/2}} \right) \left(\frac{\langle u\rangle^\delta}{t\langle u\rangle}\right)^{1-\delta^2/100}  \left( \frac{\langle u\rangle^\delta}{t^{2}}\right)^{1-\delta^2/100}       (s-t) d\theta du du'dt\\
&\le
\frac{C\epsilon^{3/2}\langle\tau\rangle^{1/2}}{s^{3-3\delta^2/100}}\int_{100}^{s-100}\iint_{\{|u'|\le 1, |u|\le 20\}}(s-t)^{-1/2}d\theta du du'dt\\
&\le \frac{C\epsilon^{3/2}\langle\tau\rangle^{1/2}s^{1/2}}{s^{3-3\delta^2/100}}<\frac{C\epsilon^{3/2}\langle\tau\rangle^\delta}{s^2}.
\end{align*}

\subsubsection{Region where $t<s/2$ and $20 < u < t^{1/2-\delta/10}$}
In this region we use the pointwise bound. We get that
\begin{align*}
C\epsilon^{3/2}&\iint_{|u'|\le 1, \tau<t<s/2, 20<u<t^{1/2-\delta/10}} \left(\frac{\sqrt{t\tau}}{s^2} \right) \left(\frac{u^\delta}{tu}\right)^{1-\delta^2/100}  \left( \frac{u^\delta}{t^{2}}\right)^{1-\delta^2/100}      t d\theta dudt d u'\\
&\le \frac{C\epsilon^{{3/2}}}{s^2}\iint_{|u'|\le 1, \tau<t<s/2, 20<u<t^{1/2-\delta/10}} \tau^{1/2}t^{-3/2+.03\delta^2}u^{-1+2\delta+.02 \delta^2} d\theta dudt d u'\\
&\le \frac{C\epsilon^{3/2}}{s^2}\int_{\tau}^{s/2} \tau^{1/2}t^{-3/2+\delta+.05\delta^2-.2\delta^2}dt\\
&\le C\epsilon^{3/2}\frac{\tau^\delta}{s^2}.
\end{align*}

\subsubsection{$t<s/2,u>t^{1/2-\delta/10},u<\tau/2$}\label{sec:uletauover2}
We use a dyadic partition in $t$ and $u$ (recall that we are assuming $u\ge 20$ and $t>100$). We obtain that in this region, we are trying to bound
\begin{align}
C&\sum_{|\alpha|+|\beta|\le 3N_0/4+10}\sum_{a,b}\iint_{|u'|\le 1, 2^a<t<2^{a+1}, 2^b<u<2^{b+1}} \left(\frac{\sqrt{t\tau}}{s^2} \right) |\overline{\partial}'\partial^\alpha \phi||\partial\partial^\beta\phi|    t d\theta dudt d u'
\end{align}
We now decompose $\partial$ into $\bar\partial$ and $\underline{L}$ and use Proposition~\ref{prop:Lbardbardbar'} (as well as the fact that in this region, $\tau-u-u'\ge\tau/3$) to get that 
\[
|\partial\partial^\beta\phi|\le C\frac{t^{1/2}}{\tau^{1/2}} |\overline\partial \partial^\beta\phi|+C\frac{t^{1/2}}{\tau^{1/2}}|\overline{\partial}' \partial^\beta\phi|.
\]
Thus we are trying to bound
\begin{align}
\frac{C}{s^2}&\sum_{|\alpha|+|\beta|\le 3N_0/4+10}\sum_{a,b}\iint_{|u'|\le 1, 2^a<t<2^{a+1}, 2^b<u<2^{b+1}} \sqrt{t\tau}\frac{t^{1/2}}{\tau^{1/2}} \left(|\overline{\partial}'\partial^\alpha \phi||\overline{\partial}'\partial^\beta\phi| +|\overline{\partial}'\partial^\alpha \phi||\overline{\partial}\partial^\beta\phi| \right) t d\theta dudt d u'
\end{align}
By Lemma~\ref{lem:endecauxcone} the $L^2$ norm of the $\overline{\partial}'\partial^\alpha\phi$ and $\overline{\partial}'\partial^\beta\phi$ over the region of integration is bounded by $C\epsilon^{3/4} 2^{-b}$. Thus, using Cauchy-Schwarz and the pointwise bounds for the good derivatives, we get
\begin{align}
C\frac{\epsilon^{3/2}}{s^2}&\sum_{a,b}2^a2^{-2b}+2^a2^{-b}\left(\iint_{|u'|\le 1, 2^a<t<2^{a+1}, 2^b<u<2^{b+1}} \left(\frac{u^{\delta}}{t^2}\right)^{2-2\delta^2}t d\theta dudt d u'\right)^{1/2}\\
&\le  C\frac{\epsilon^{3/2}}{s^2}\sum_{a,b} 2^{a-2b}+2^a2^{-b}\left((2^a)(2^b) 2^{2\delta b}2^{-3a}2^{4\delta^2 a}\right)^{1/2}\\
&\le C\frac{\epsilon^{3/2}}{s^2} \sum_{a,b} 2^{a-2b}+2^{2\delta^2 a}2^{-b(1/2-\delta)}
\end{align}
We now use that  $\tau>2u>Ct^{1/2-\delta/10}\ge C\cdot 2^{a(1/2-\delta/10)}$, that $b\ge a(1/2-\delta/10)-5$, and that $a, b \le C \log(\tau)$ to get that
\[
C\frac{\epsilon^{3/2}}{s^2} \sum_{a,b} 2^{a-2b}+2^{2\delta^2 a}2^{-b(1/2-\delta)}\le C\frac{\epsilon^{3/2}}{s^2}\sum_{a,b}2^{\delta a/4}+2^{-b/2+2\delta b}\le C\frac{\epsilon^{3/2}\tau^\delta}{s^2}
\]
\subsubsection{$t<s/2,u>t^{1/2-\delta/10},\tau/2\le u\le \tau-\tau^{1-0.6\delta}$}
In this region, $u,\tau$ are comparable and $t<\tau^{2+\delta/2}$. Similarly to Section~\eqref{sec:uletauover2}, we note that $|\tau-u-u'|>|\tau-u|/2$ and thus obtain using Proposition~\ref{prop:Lbardbardbar'} that
\[
|\partial\partial^\beta\phi|\le C\frac{t^{1/2}}{(\tau-u)^{1/2}} |\overline\partial \partial^\beta\phi|+C\frac{t^{1/2}}{(\tau-u)^{1/2}}|\overline{\partial}' \partial^\beta\phi|.
\]
Thus, decomposing the null form relative to the frame adapted to $\overline{\partial}'$, we need to bound
\begin{align}
C&\sum_{|\alpha|+|\beta|\le 3N_0/4+10}\iint_{t<s/2, u>t^{1/2-\delta/10},\tau/2\le u\le \tau-\tau^{1-0.6\delta},|u'|\le 1} \left(\frac{\sqrt{t\tau}}{s^2} \right) |\overline{\partial}'\partial^\alpha \phi||\partial\partial^\beta\phi|    t d\theta dudt d u'\nonumber\\
&\le C\sum_{|\alpha|+|\beta|\le 3N_0/4+10}\iint_{t<s/2, u>t^{1/2-\delta/10},\tau/2\le u\le \tau-\tau^{1-0.6\delta},|u'|\le 1}\nonumber\\
&\qquad\left(\frac{t\tau^{1/2}}{s^2(\tau-u)^{1/2}} \right) \left(|\overline{\partial}'\partial^\alpha \phi|||\overline{\partial}'\partial^\beta\phi|+|\overline{\partial}'\partial^\alpha \phi|||\overline{\partial}\partial^\beta\phi|\right)    t d\theta dudt d u' \label{tobound:currentlycase2.3}
\end{align}

For the first term, we note that by Lemma~\ref{lem:endecauxcone}, the $L^2$ norm of $\overline{\partial}'\partial^\alpha \phi$  and of $\overline{\partial}'\partial^\beta \phi$ is at most $C\epsilon^{3/4}\tau^{-1}$ and we bound the prefactor pointwise and we use that $t<\tau^{2+\delta/2}$ to get
\begin{equation}
    \begin{aligned}
    C\iint_{t<s/2, u>t^{1/2-\delta/10},\tau/2\le u\le \tau-\tau^{1-0.6\delta},|u'|\le 1}\frac{t\tau^{1/2}}{s^2(\tau-u)^{1/2}} |\overline{\partial}'\partial^\alpha \phi|||\overline{\partial}'\partial^\beta\phi|td\theta du dt d u'
    \\ \le C\epsilon^{3/2}\frac{\tau^{2+\delta/2}\tau^{1/2}}{s^2\tau^{1/2-0.3\delta}}\tau^{-2}<C\epsilon^{3/2}\frac{\tau^{\delta}}{s^2}
    \end{aligned}
\end{equation}
Now we bound the second term of \eqref{tobound:currentlycase2.3}. By Lemma~\ref{lem:endecauxcone}, we know that the $L^2$ norm of $\overline{\partial}'\partial^\alpha\phi$ is at most $\epsilon^{3/4}\tau^{-1}$. We calculate the square of $L^2$ norm of the remaining factor over the same region of integration (but not explicitly writing it for ease of notation) using the pointwise bounds (and using that $u$ and $\tau$ are comparable, and that $t<C\tau^{2+\delta}$)
\begin{align*}
C\epsilon^{3/2}&\iint \frac{t^2\tau}{s^4(\tau-u)}\left(\frac{u^\delta}{t^2}\right)^{2-\delta^2}  t d\theta dudt d u'\\
&\le \frac{C\epsilon^{3/2}\tau^{1+2\delta}}{s^4}\iint \frac{t^{-1+2\delta^2}}{\tau-u} d\theta dudt d u'\\
&\le \frac{C\epsilon^{3/2}\tau^{1+2\delta}\log\tau}{s^4}\int_{\tau-2}^{C\tau^{2+\delta}}  t^{-1+2\delta^2}dt\le\frac{C\epsilon^{3/2}\tau^{1+2\delta+5\delta^2}\log\tau}{s^4}
\end{align*}
so the $L^2$ norm of the other factor is at most
\[
\frac{C\epsilon^{3/4}\tau^{1/2+2\delta}}{s^2}
\]
and the second term in \eqref{tobound:currentlycase2.3} is bounded by
\[
\epsilon^{3/4}\tau^{-1}\frac{C\epsilon^{3/4}\tau^{1/2+2\delta}}{s^2}<C\epsilon^{3/2}\frac{\tau^{\delta}}{s^2}.
\]

\subsubsection{$t<s/2,u>t^{1/2-\delta/10},\tau-u<\tau^{1-0.6\delta}$} \label{sec:tsmallulargetau-usmall} 
In this region, $u$ and $\tau$ are comparable and $t<\tau^{2+\delta/2}$. Decomposing the null form in a frame adapted to $\overline{\partial}'$, we use pointwise bounds for $|\partial\partial^{\beta}\phi|$ to bound
\begin{align*}
&\sum_{|\alpha|+|\beta|\le 3N_0/4+10}\iint_{|u'|\le 1, 100<t<s/2,u>t^{1/2-\delta/10},\tau-u<\tau^{1-0.6\delta}}\frac{\sqrt{t\tau}}{s^2} |\overline{\partial}'\partial^\alpha\phi||\partial\partial^\beta\phi| td\theta du dt d u'\\
&\le C\epsilon^{3/4}\sum_{|\alpha|\le 3N_0/4+10}\iint_{|u'|\le 1, 100<t<s/2,u>t^{1/2-\delta/10},\tau-u<\tau^{1-0.6\delta}}\frac{\sqrt{t\tau}}{s^2} |\overline{\partial}'\partial^\alpha\phi|\left(\frac{ u^\delta}{tu}\right)^{1-\delta^2} td\theta du dt d u'\\
&\le C\epsilon^{3/4}\tau^{-1/2+\delta+\delta^2}s^{-2}\sum_{|\alpha|\le 3N_0/4+10}\iint_{|u'|\le 1, 100<t<s/2,u>t^{1/2-\delta/10},\tau-u<\tau^{1-0.6\delta}}t^{-1/2+\delta^2} |\overline{\partial}'\partial^\alpha\phi| td\theta du dt d u'.
\end{align*}
We now apply Cauchy-Schwarz and use Lemma~\ref{lem:endecauxcone} for the second factor to get that
\begin{align*}
 C\epsilon^{3/4}&\tau^{-1/2+\delta+\delta^2}s^{-2}\iint_{|u'|\le 1, 100<t<s/2,u>t^{1/2-\delta/10},\tau-u<\tau^{1-0.6\delta}}t^{-1/2+\delta^2} |\overline{\partial}'\partial^\alpha\phi| td\theta du dt\\
 &\le  C\epsilon^{3/4}\tau^{-1/2+\delta+\delta^2}s^{-2}\left(\iint t^{-1+2\delta^2} td\theta du dt\right)^{1/2}\left(\iint |\overline{\partial}'\partial^\alpha\phi|^2 td\theta du dt\right)^{1/2}\\
 &\le  C\epsilon^{3/2}\tau^{-1/2+\delta+\delta^2}s^{-2}\left( (\tau^{1-0.6\delta})\int_{100}^{\tau^{2+\delta/2}} t^{2\delta^2} dt\right)^{1/2}\tau^{-1}\\
 &\le C\epsilon^{3/2}\tau^{-1/2+\delta+\delta^2}s^{-2}\left(\tau^{1-0.6\delta+(2+\delta/2)(1+2\delta^2)}\right)^{1/2}\tau^{-1}<C\epsilon^{3/2}\frac{\tau^{\delta}}{s^2},
\end{align*}
where we stopped writing the region of integration explicitly in some of the terms above.

\subsubsection{Case where $\tau<s^{1-10\delta}$, region where $t>s/2$} \label{sec:tlargetausmall}
Here we use pointwise bootstrap assumptions to get that we need to bound
\begin{align*}
C&\sum_{|\alpha|+|\beta|\le 3N_0/4+10}\iint_{s - 100 > t > s / 2, |u'| \le 10} \frac{\tau^{1/2}}{(s-t)^{3/2}}|\overline\partial\partial^\alpha\phi||\partial\partial^\beta\phi|(s-t)d\theta du dt d u'\\
&\le C\epsilon^{3/2}\sum_{|\alpha|+|\beta|\le 3N_0/4+10}\iint_{s - 100 > t > s / 2, |u'| \le 10} \frac{\tau^{1/2}}{(s-t)^{3/2}}\left(\frac{u^\delta}{s^2}\right)^{1-\delta^2}  \left(\frac{u^{\delta}}{us}\right)^{1-\delta^2}  (s-t)d\theta du dt d u'\\
&\le\frac{C\epsilon^{3/2}}{s^{3-3\delta^2}}\int^{s-100}_{s/2}\frac{\tau^{1/2+2\delta+\delta^2}}{(s-t)^{1/2}}  dt\\
&\le \frac{C\epsilon^{3/2}\tau^{1/2+2\delta+\delta^2}}{s^{5/2-3\delta^2}}\le\frac{C\epsilon^{3/2}}{s^2}\le C\epsilon^{3/2}\frac{\tau^\delta}{s^2}.
\end{align*}

\subsubsection{Case where $\tau\ge s^{1-10\delta}$, region where $t\ge s-\tau/4$}\label{sec:taubigfirst}
In this region, we have that $\tau$ and $u$ are of comparable size. We will ignore that the derivatives of $\psi$ are only in certain directions and just use the bound
\[
|\partial \psi|\le\frac{1}{s-t}.
\]

Then we use pointwise bounds for $|\partial\partial^{\beta}\phi|$ and see that we need to bound
\begin{align*}
C&\sum_{|\alpha|+|\beta|\le 3N_0/4+10}\iint_{s-100>t>s-\tau/2,|u'| \le 10} \frac{1}{s-t}|\overline{\partial}'\partial^\alpha\phi||\partial\partial^\beta\phi|dx dt\\
&\le C\epsilon^{3/4}\sum_{|\alpha|+|\beta|\le 3N_0/4+10}\iint_{s-100>t>s-\tau/2,|u'| \le 10} \frac{1}{s-t}|\overline{\partial}'\partial^\alpha\phi|\left(\frac{u^{\delta}}{us}\right)^{1-\delta^2}dx dt\\
&\le C\epsilon^{3/4}\frac{\tau^{\delta+\delta^2}s^{\delta^2}}{\tau s}\sum_{|\alpha|+|\beta|\le 3N_0/4+10}\iint_{s-100>t>s-\tau/2,|u'| \le 10} \frac{1}{s-t}|\overline{\partial}'\partial^\alpha\phi|dx dt.
\end{align*}

We now apply Cauchy-Schwarz and use Lemma~\ref{lem:endecauxcone} for the second factor to get that

\begin{align*}
&\epsilon^{3/4}\frac{\tau^{\delta+\delta^2}s^{\delta^2}}{\tau s}\sum_{|\alpha|+|\beta|\le 3N_0/4+10}\iint_{s-100>t>s-\tau/2,|u'| \le 10} \frac{1}{s-t}|\overline{\partial}'\partial^\alpha\phi|dx dt\\
&\le C\epsilon^{3/4}\frac{\tau^{\delta+\delta^2}s^{\delta^2}}{\tau s}\sum_{|\alpha|+|\beta|\le 3N_0/4+10}\left(\iint_{s - 100 > t > s - \tau / 2, |u'| \le 10} \frac{1}{(s-t)^2}dx dt\right)^{1/2}\left(\iint |\overline{\partial}'\partial^\alpha\phi|^2dx dt\right)^{1/2}\\
&\le C\epsilon^{3/2}\frac{\tau^{\delta+\delta^2}s^{\delta^2}}{\tau s}\sum_{|\alpha|+|\beta|\le 3N_0/4+10}\left(\int_{s-\tau/2}^{s-100} dt\right)^{1/2}\tau^{-1}\\
&\le C\epsilon^{3/2}\frac{\tau^{\delta+\delta^2}s^{\delta^2}}{\tau s}\tau^{1/2}\tau^{-1}<C\epsilon^{3/2}\frac{\tau^\delta}{s^2},
\end{align*}
where we note that we stopped writing the region of integration explicitly in the above.

\subsubsection{Case where $\tau\ge s^{1-10\delta}$, region where $s/2<t<s-\tau/4$ and $u>\tau^{4/5}$}
In this region, as in Section~\ref{sec:taubigfirst}, we use pointwise bounds on $|\partial\partial^{\beta}\phi|$, then Cauchy-Shwarz and  Lemma~\ref{lem:endecauxcone} to obtain the bound
\begin{align*}
&\sum_{|\alpha|+|\beta|\le 3N_0/4+10}\iint_{s-\tau/4>t>s/2,u>\tau^{4/5},|u'| \le 10} \frac{\tau^{1/2}}{(s-t)^{3/2}}|\overline{\partial}'\partial^\alpha\phi||\partial\partial^\beta\phi|dx dt\\
&\le C \epsilon^{{3 \over 4}} \sum_{|\alpha|+|\beta|\le 3N_0/4+10}\iint_{s-\tau/4>t>s/2,u>\tau^{4/5},|u'| \le 10} \frac{\tau^{1/2}}{(s-t)^{3/2}}\left(\frac{u^{\delta}}{us}\right)^{1-\delta^2}|\overline{\partial}'\partial^\alpha\phi|dx dt\\
&\le C \epsilon^{{3 \over 4}} \sum_{|\alpha|+|\beta|\le 3N_0/4+10}\frac{\tau^{{1 \over 2}} s^{\delta^2}}{s}(\tau^{{4 \over 5}})^{-1+\delta+\delta^2}\left(\iint_{s-\tau/4>t>s/2,u>\tau^{4/5},|u'| \le 10}\frac{1}{(s-t)^3}dxdt\right)^{{1 \over 2}} \left(\iint |\overline{\partial}'\partial^\alpha\phi|^2 dxdt\right)^{{1 \over 2}}\\
&\le C\epsilon^{3/2}\frac{s^{\delta^2}}{s}\tau^{-0.25}\left(\int_{s/2}^{s-\tau/4}\frac{1}{(s-t)}dt\right)^{1/2}\tau^{-4 / 5}<C\epsilon^{3/2}\frac{\tau^{\delta}}{s^2}.
\end{align*}
\subsubsection{Case where $\tau \ge s^{1 - 10 \delta}$, $s / 2\le t\le s-\tau/4$, and $u \le t^{1 / 2 - \delta / 10}$}
In this region we use the pointwise bound. We get that
\begin{align*}
C\epsilon^{3/2}&\iint_{|u'|\le 1, s / 2\le t\le s-\tau/4, u \le t^{1 / 2 - \delta / 10}} \left(\frac{\tau^{1/2}}{(s-t)^{3/2}} \right) \left(\frac{u^\delta}{tu}\right)^{1-\delta^2/100}  \left( \frac{u^\delta}{t^{2}}\right)^{1-\delta^2/100}      (s-t) d\theta dudt d u'\\
&\le \frac{C\epsilon^{3/2}\tau^{1/2}s^{0.03\delta^2}}{s^3}\iint_{|u'|\le 1, s / 2\le t\le s-\tau/4, u \le t^{1 / 2 - \delta / 10}} (s-t)^{-1/2}u^{-1+2\delta+0.01\delta^2} d\theta du dt d u'\\
&\le \frac{C\epsilon^{3/2}\tau^{1/2}s^{0.03\delta^2}}{s^3}\int_{s/2}^{s-\tau/4} (s-t)^{-1/2}s^{(1/2-\delta/10)(2\delta+0.01\delta^2)} dt\\
&\le \frac{C\epsilon^{3/2}\tau^{1/2}s^{0.03\delta^2}}{s^3}s^{1/2+(1/2-\delta/10)(2\delta+0.01\delta^2)} dt<C\epsilon^{3/2}\frac{\tau^{\delta}}{s^2}.
\end{align*}

\subsubsection{Case where $\tau \ge s^{1 - 10 \delta}$, $s/2\le t \le s -\tau/4$, and $t^{1 / 2 - \delta / 10} \le u \le \tau^{4/5}$}\label{subsec:penultimate}
We obtain using Proposition~\ref{prop:Lbardbardbar'} that
\[
|\partial\partial^\beta\phi|\le C\frac{(s-t)^{1/2}}{\tau^{1/2}} |\overline\partial \partial^\beta\phi|+C\frac{(s-t)^{1/2}}{\tau^{1/2}}|\overline{\partial}' \partial^\beta\phi|\le Cs^{5\delta}(|\overline\partial \partial^\beta\phi|+|\overline{\partial}' \partial^\beta\phi|).
\]
Thus, using pointwise bounds for $|\overline\partial \partial^\alpha\phi|,|\overline\partial \partial^\beta\phi|$, we obtain the bound
\begin{align*}
&\iint_{s/2\le t \le s-\tau/4, t^{1 / 2 - \delta / 10} \le u \le \tau^{4/5}, |u'| \le 1} \left(\frac{\tau^{1/2}}{(s-t)^{3/2}} \right)|\overline\partial\partial^\alpha\phi||\partial\partial^\beta\phi|d x d t\\
&\qquad\le Cs^{5\delta}\tau^{1/2}\iint_{s/2\le t \le s-\tau/4, t^{1 / 2 - \delta / 10} \le u \le \tau^{4/5}, |u'| \le 1}\frac{1 }{(s-t)^{3/2}}\left(|\overline\partial\partial^\alpha\phi||\overline\partial\partial^\beta\phi|+|\overline \partial\partial^\alpha\phi||\overline{\partial}'\partial^\beta\phi|\right)d x d t\\
&\qquad\le Cs^{5\delta}\tau^{1/2}\epsilon^{3/2}\iint_{s/2\le t \le s-\tau/4, t^{1 / 2 - \delta / 10} \le u \le \tau^{4/5}, |u'| \le 1}\frac{1}{(s-t)^{3/2}}\left(\frac{u^\delta}{s^2}\right)^{2-\delta^2}dx dt\\
&\qquad\qquad +Cs^{5\delta}\tau^{1/2}\epsilon^{3/4}\iint_{s/2\le t \le s-\tau/4, t^{1 / 2 - \delta / 10} \le u \le \tau^{4/5}, |u'| \le 1} \frac{1 }{(s-t)^{3/2}}\left(\frac{u^\delta}{s^2}\right)^{1-\delta^2/2}|\overline{\partial}'\partial^\beta\phi| dxdt\\
&\qquad\le Cs^{5\delta}\tau^{1/2}\epsilon^{3/2}\iint_{s/2\le t \le s-\tau/4, t^{1 / 2 - \delta / 10} \le u \le \tau^{4/5}, |u'| \le 1}\frac{1}{(s-t)^{3/2}}\left(\frac{u^\delta}{s^2}\right)^{2-\delta^2}(s-t) d\theta du dt d u'\\
&\qquad\qquad +Cs^{5\delta}\tau^{1/2}\epsilon^{3/4}\left(\iint_{s/2\le t \le s-\tau/4, t^{1 / 2 - \delta / 10} \le u \le \tau^{4/5}, |u'| \le 1} \frac{1 }{(s-t)^{3}}\left(\frac{u^\delta}{s^2}\right)^{2-\delta^2} (s-t) d\theta d u d t  d u'\right)^{1/2}\\
&\qquad\qquad\qquad\qquad\left(\iint_{s/2\le t \le s-\tau/4, t^{1 / 2 - \delta / 10} \le u \le \tau^{4/5}, |u'| \le 1}|\overline{\partial}'\partial^\beta\phi|^2 dx dt\right)^{1/2}.
\end{align*}
By Lemma~\ref{lem:endecauxcone}, we know that the $L^2$ norm of $\overline{\partial}'\partial^\beta\phi$ is at most $C\epsilon^{3/4}t^{-(1/2-\delta/10)}\le C\epsilon^{3/4}s^{-(1/2-\delta/10)}$, so
\begin{align*} 
&\iint_{s/2\le t \le s-\tau/4, t^{1 / 2 - \delta / 10} \le u \le \tau^{4/5}, |u'| \le 1} \left(\frac{\tau^{1/2}}{(s-t)^{3/2}} \right)|\overline\partial\partial^\alpha\phi||\partial\partial^\beta\phi|d x d t\\
&\qquad\le Cs^{5\delta}\tau^{1/2}\epsilon^{3/2}\frac{s^{2\delta^2}\tau^{4/5(1+2\delta)}}{s^4}\int_{s/2}^{s-\tau/4}(s-t)^{-1/2}dt\\
&\qquad\qquad+ Cs^{5\delta}\tau^{1/2}\epsilon^{3/2}\left(\frac{s^{2\delta^2}\tau^{4/5(1+2\delta)}}{s^4}\int_{s/2}^{s-\tau/4}(s-t)^{-2}dt\right)^{1/2}s^{-(1/2-\delta/10)}\\
&\qquad\le Cs^{5\delta}\tau^{1/2}\epsilon^{3/2}\frac{s^{2\delta^2}\tau^{4/5(1+2\delta)}}{s^4}s^{1/2}+ Cs^{5\delta}\tau^{1/2}\epsilon^{3/2}\left(\frac{s^{2\delta^2}\tau^{4/5(1+2\delta)-1}}{s^4}\right)^{1/2}s^{-(1/2-\delta/10)}\\
&\qquad<C\epsilon^{3/2}\frac{\tau^{\delta}}{s^2}.
\end{align*}
\subsubsection{Case where $\tau\ge s/{10}$}\label{sec:taularge}
We need to bound \eqref{firstintegraltobound2v2}, which comes out to bounding an integral over a backward cone of thickness 1 given certain bounds on $\psi$ and certain bootstrap assumptions on $\phi$ in this cone. We can perform a Lorentz boost to make it so the tip of the cone is at a point where $s/100\le \tau\le s/20$, and the same bootstrap assumptions hold (up to constant factors). The integral we need to bound is new, but it is equivalent up to constant factors to the integral we have already bounded in the previous subcases, so we are done.

\subsection{Recovery of  bootstrap assumption~\ref{ba:badder}}\label{sec:badderivrecovery}
This bootstrap assumption contains the pointwise estimates for $\phi$ and up to ${3 N_0 \over 4} + 1$ generic derivatives.

Proceeding as in Section~\ref{sec:goodderivrecovery} but using Proposition~\ref{prop:badderpointwise} instead of Proposition~\ref{prop:goodderpointwise} gives us that
\[
|\partial \partial^\alpha \phi| (s,\tau,0,0) \le {C \over (1 + s)} \chi_{|\tau| \le 5} \Vert \partial \partial^\alpha \phi \Vert_{H^{{3 N_0 \over 4} + 10} (\Sigma_0)} + C \sup_{|\beta| \le 10} \sup_{\psi_1} \int_0^s \int_{\Sigma_t} |\partial^\beta \partial^\alpha (f g(d \phi,d \phi))| |\partial \psi| d x d t.
\]
Now, we have that
\[
{C \over (1 + s)} \chi_{|\tau| \le 5} \Vert \partial \partial^\alpha \phi \Vert_{H^{{3 N_0 \over 4} + 10} (\Sigma_0)} \le {C \epsilon \over (1 + s)} \chi_{\tau \le 5},
\]
meaning that this term is appropriately controlled in order to recover the bootstrap assumption. All that remains is to control the error integral given by
\begin{equation}\label{firstintegraltobound}
\sup_{|\beta| \le 10} \sup_{\psi_1} \int_0^s \int_{\Sigma_t} |\partial^\beta \partial^\alpha (f g(d \phi,d \phi))| |\partial \psi| d x d t.
\end{equation}

\subsubsection{$s\le 200$, $t\le 100$, or $s-t\le 100$}\label{sec:somethingsmallbadder}

In the case that $s\le 200$ or $s>200$ and $t\le 100$, we use that
\[
|\partial\psi|\le\frac{C}{1+s} \chi_{|u'| \le 1}
\]
and put both derivatives of $\phi$ in the integrand of \eqref{firstintegraltobound} in $L^2$  and use bootstrap assumption~\eqref{ba:totalen} to get the bound
\[
\frac{C\epsilon^{3/2}}{(1+s) (1 + \tau)^{1 - \delta}}.
\]

In the case that $s-t\le 100$, we put $|\partial\psi|$ in $L^2$, put whichever factor of $\phi$ gets hit by more derivatives in $L^2$ and bound it using bootstrap assumptions~\eqref{ba:totalen}, and put whichever factor of $\phi$ gets hit by fewer derivatives in $L^\infty$ and bound it using bootstrap assumptions~\eqref{ba:badder}. This then gives us an upper bound of
\[
C(C \epsilon^{3/4})\frac{C\epsilon^{3/4}}{(1+s) (1 + \tau)^{1 - \delta}}\le \frac{C\epsilon^{3/2}}{(1 + s) (1 + \tau)^{1 - \delta}}.
\]

Thus we only need to bound
\begin{equation}\label{firstintegraltobound2v2}
C  \sup_{|\gamma|\le 3N_0/4+10}\int_{100}^{s-100} \int_{\Sigma_t} \partial^{\gamma}g(d \phi,d \phi)|\partial \psi|dxdt.
\end{equation}

\subsubsection{All other regions}
Here, we use the upper bound of $\frac{C}{s-t}$ for $|\partial \psi|$. When $t<s/2$, we note that 
\[
\frac{1}{1+s-t}\le C\left(\frac{s}{\langle\tau\rangle}\right)\frac{\sqrt{t\langle\tau\rangle}}{s^2},
\]
pull the prefactor out of the integral, and then observe that we can just use the bound on the integral we obtain in different regions in Sections~\ref{sec:usmall} through \ref{sec:tsmallulargetau-usmall}. We note that this inequality uses that $|u'| \le 1$.

When $t\ge s/2$, we similarly note that 
\[
\frac{1}{s-t}\le C\left(\frac{s}{\langle\tau\rangle}\right)\frac{\langle\tau\rangle^{1/2}}{(s-t)^{3/2}},
\]
pull the prefactor out of the integral, and then observe that we can just use the bound on the integral we obtain in different regions in Sections~\ref{sec:tlargetausmall} through \ref{sec:taularge}. This completes the recovery of the bootstrap assumptions found in Section~\ref{sec:BootstrapAssumptions} and, thus, completes the proof of Theorem~\ref{thm:main}.

\bibliographystyle{abbrv}
\bibliography{references}

\end{document}